\documentclass[12pt,letterpaper,headinclude,footinclude,fleqn,reqno]{amsart}                 
\usepackage[T1]{fontenc}                   
\usepackage[utf8]{inputenc}                 
\usepackage[english]{babel}       
\usepackage{graphicx}                      % 
\usepackage[font=small]{quoting}            %  

\usepackage{float}
\usepackage{caption}  
\usepackage[colorlinks=true,linkcolor=blue]{hyperref}

\usepackage{amsmath}

\usepackage[top=1in,bottom=1.5in,left=1.2in,right=1.2in]{geometry}

\usepackage{verbatim}
\usepackage{color}
\usepackage{picture}
\usepackage{graphicx}
\usepackage{float}
\usepackage{comment}
\usepackage{hyperref}
\hypersetup{colorlinks,linkcolor={blue},citecolor={red},urlcolor={red}}

\def\n{\nabla}
\def\wt{\widetilde}

\newcommand{\ol}{\overline}

\def\<{\langle}
\def\>{\rangle}

\def\O{\Omega}

\def\S{\Sigma}
\def\n{\nabla}
\def\a{\alpha}
\def\b{\beta}
\def\n{\nabla}
\def\l{\lambda}

\def\O{\Omega}
\def\p{\partial}

\def\a{\alpha}
\def\b{\beta}
\def\g{\gamma}

\def\k{\kappa}
\def\l{\lambda}
\def\L{\mathcal{L}}
\def\s{\sigma}

\def\ov{\overline}

\def\n{\nabla}
\def\<{\langle}
\def\>{\rangle}
\def\div{{\rm div}}
\def\RR{\mathbb{R}}

\def\SS{\mathbb{S}}
\def\HH{\mathbb{H}}

\def\BB{\mathbb{B}}

\def\W{\mathcal{W}}

\def\A{\mathcal{A}}
\def\H{\mathcal{H}}
\def\C{\mathcal{C}}
\def\T{\mathcal{T}}
\def\wh{\widehat}

\def\p{\partial}

\def\ov{\overline}

\newtheorem{thm}{Theorem}[section] 
\newtheorem{lem}[thm]{Lemma} 
\newtheorem{cor}[thm]{Corollary} 
\newtheorem{prop}[thm]{Proposition} 
 
\newtheorem{defn}[thm]{Definition} 
\theoremstyle{definition} 
\newtheorem{rem}[thm]{Remark} 
\theoremstyle{remark} 

\theoremstyle{question}

\usepackage{graphicx}

\numberwithin{equation}{section}

\makeatletter
\@namedef{subjclassname@2020}{\textup{2020} Mathematics Subject Classification}
\makeatother

\begin{document}
\setlength{\baselineskip}{1.2\baselineskip}

\title[Alexandrov-Fenchel inequalities in hyperbolic space]
{Alexandrov-Fenchel inequalities \\ for
capillary hypersurfaces in hyperbolic space}

\author[X. Mei]{Xinqun Mei}\address[X. Mei]{Mathematisches Institut, Albert-Ludwigs-Universit\"{a}t Freiburg, Freiburg im Breisgau, 79104, Germany}\email{xinqun.mei@math.uni-freiburg.de}

\author[L. Weng]{Liangjun Weng}\address[L. Weng]{Dipartimento di Matematica, Università di Pisa, Pisa, 56127, Italy.  Dipartimento Di Matematica, Università degli Studi di Roma "Tor Vergata", Roma, 00133, Italy}\email{liangjun.weng@uniroma2.it}

	% Please provide  minimum  5 keywords.
	
\maketitle 

\begin{center}   \scriptsize{\textit{Dedicated to Professor Guofang Wang on the occasion of his 60th birthday}}\end{center}

\begin{abstract} 
In this article, we first introduce the quermassintegrals for compact hypersurfaces with capillary boundaries in hyperbolic space from a variational viewpoint, and then we solve an isoperimetric type problem in hyperbolic space. By constructing a new locally constrained inverse curvature flow, we obtain the Alexandrov-Fenchel inequalities for convex capillary hypersurfaces in hyperbolic space. This generalizes a theorem of Brendle-Guan-Li \cite{BGL} for convex closed hypersurfaces in hyperbolic space.  \end{abstract}  

%\tableofcontents

    \section{Introduction}
   Let $\Omega$ be a  convex body in hyperbolic space $\HH^{n+1}$ with boundary $\partial \Omega$. The $k$-th quermassintegral of $\Omega$, denoted as $\mathcal{W}_{k}(\Omega)$,  is defined as the volume of the set of totally geodesic $k$-dimensional subspaces that intersect $\Omega$  (see e.g. \cite[Part \uppercase\expandafter{\romannumeral4}]{Santalo} or \cite{Sola}).   In particular, \begin{eqnarray}\label{quer for closed0}
        \mathcal{W}_{0}(\Omega)=|\Omega|, ~~~~\W_1(\O)=\frac{1}{n+1}|\p \O|. %\W_{n+1}(\O)= \frac{|\SS^n|}{n+1}.
   \end{eqnarray}If further assume that $\p\O$ is smooth (say at least $C^2$), then the quermassintegrals $\W_k(\O)$ and the curvature integrals are related  
(see e.g. \cite[Proposition~7]{Sola}) by
\begin{eqnarray}\label{quer for compact hypersurface}
\begin{aligned}
   \mathcal{W}_{k+1}(\Omega)=\frac{1}{n+1}\int_{\partial\Omega}H_{k}dA-\frac{k}{n+2-k}\mathcal{W}_{k-1}(\Omega), ~1\leq k\leq n-1.
    \end{aligned}
\end{eqnarray}
Here $H_{k}$ is the normalized $k$-th mean curvature of $\partial\Omega\subset \HH^{n+1}$, see Section \ref{sec2.1} for precise definition. The above quermassintegrals possess a nice variational structure (see e.g. \cite[Proposition 3.1]{WX2014} or \cite[Section 4]{BC97}):
\begin{eqnarray}\label{variation of quer}
    \frac{d}{dt} \mathcal{W}_{k}(\Omega_{t}) =\frac{n+1-k}{n+1}\int_{\partial\Omega_{t}}fH_{k}dA_{t}, \quad 0\leq k\leq n+1, \end{eqnarray}
for any normal variation along $\p \O_t$ with the speed function $f$. Furthermore,  the Alexandrov-Fenchel inequalities involving the quermassintegrals $\W_k(\O)$ in hyperbolic space have attracted wide attention in recent decades, it states 
\begin{eqnarray}\label{Alex-Fen}         \mathcal{W}_{k}(\Omega)\geq f_{k}\circ f^{-1}_{l}\left(\mathcal{W}_{l}(\Omega)\right),~ 0\leq l<k\leq n,   \end{eqnarray}where $f_{k}:[0,\infty)\rightarrow \RR_{+}$ is the monotone function defined by $f_{k}(\rho):=\mathcal{W}_{k}(B_{\rho})$, with $B_\rho$ being the geodesic ball of radius $\rho$ in $\HH^{n+1}$, and $f_{l}^{-1}$ being the inverse function of $f_{l}$. Moreover, the equality holds in \eqref{Alex-Fen} if and only if $\p\Omega$ is a geodesic sphere. In \cite{WX2014}, Wang-Xia studied a globally constrained quermassintegral preserving flow, given by one parameter family of embedded hypersurface $x(\cdot,t): M^n\times[0,T)\to \HH^{n+1}$ satisfying
 \begin{eqnarray}\label{wang-xia flow}
     \p_t x=\left(\frac{\int_{\partial\Omega_{t}}H_{k}^{\frac{1}{k-l}}H_{l}^{1-\frac{1}{k-l}}dA_{t}}{\int_{\partial\Omega_{t}}H_{l}dA_{t}}-\left(\frac{H_{k}}{H_{l}} \right)^{\frac{1}{k-l}}\right)\nu, \quad 0\leq l<k\leq n.
 \end{eqnarray}where $\nu=\nu(\cdot,t)$ is the unit outward normal of $x(\cdot,t)$. The flow \eqref{wang-xia flow} preserves $\W_l(\O_t)$ while decreases $\W_k(\O_t)$, then they established the Alexandrov-Fenchel inequalities \eqref{Alex-Fen} for $h$-convex domain $\O$ in $\HH^{n+1}$ (cf. \cite[Theorem~1.1]{WX2014}). 
Here a domain $\Omega\subset \HH^{n+1}$ is referred to as $h$-convex if all the principal curvatures of its boundary $\partial \Omega$ are greater or equal to $1$. In other words, the minimum value of $\W_k(\O)$ among all the $h$-convex closed hypersurfaces $\p\O$ in $\HH^{n+1}$ with a fixed value $\W_l(\O)$ is achieved by the geodesic sphere. This solves a natural isoperimetric type problem for closed hypersurfaces in $\HH^{n+1}$. In particular, when  $k=1$ and $l=0$, \eqref{Alex-Fen} reduces to the classical isoperimetric inequality in $\HH^{n+1}$ relating the area and volume, which was established by Schmidt in \cite{Schm}.

On the other hand, the hyperbolic space can be viewed as a warped product space $\mathbb{H}^{n+1}=[0,\infty)\times \mathbb{S}^{n}$, equipped with the metric $$\bar{g}=d\rho^{2}+\phi^{2}(\rho)\sigma,$$ where $\phi(\rho)=\sinh \rho$ and $\sigma$ is the standard spherical metric on $\SS^{n}$. Based on the Minkowski formula (see e.g. Guan-Li \cite[Proposition 2.5]{GL2014}) for the closed hypersurface $M:=\p \O\subset \HH^{n+1}$
\begin{eqnarray}\label{mink formula-closed}
    \int_{M} (\phi' H_{k-1} -v H_k )dA=0,~~~1\leq k \leq n,
\end{eqnarray}where $\phi' =\cosh \rho$ and $v$ is the support function of $M$ as $$v=\bar{g}(\phi(\rho) \partial_{\rho}, \nu).$$% is the support function of $M$ and $\phi' =\cosh \rho$. 
In \cite{BGL}, Brendle-Guan-Li designed a locally constrained inverse curvature flow  as 
\begin{eqnarray}\label{BGL-flow}
    \p_t x=\left(\frac{\phi' }{F}-v\right)\nu,
\end{eqnarray}
where $F=\frac{H_{k}}{H_{k-1}}$. Along the flow \eqref{BGL-flow}, when the evolving hypersurfaces are  $k$-convex, then 
the $k$-th quermassintegral $\mathcal{W}_{k}(\Omega_{t})$ is preserved and the $(k+1)$-th quermassintegral $\mathcal{W}_{k+1}(\Omega_{t})$ is non-increasing with respect to the time $t\geq 0$.  Here a  smooth hypersurface $M\subset \HH^{n+1}$ is  $k$-convex for some $1\leq k\leq n$ means that its principal
curvatures $\kappa:=(\k_1,\cdots, \k_n)\in\Gamma_{k}$, see \eqref{2.4}. A smooth hypersurface $M\subset \HH^{n+1}$  is called {star-shaped} if its support function $v$ is positive everywhere on $M$.  In \cite[Theorem~1.3]{BGL}, Brendle-Guan-Li established the long-time existence and convergence of flow \eqref{BGL-flow} under two cases: either the initial closed hypersurface $M_0$ is strictly convex and $k=n$ or $M_0$ is star-shaped, $k$-convex and satisfying a gradient bound condition. As a consequence, the inequalities \eqref{Alex-Fen}  holds for $k=n$ and $0\leq l\leq n-1$ provided that $\p\Omega$ is convex. Recently, Hu-Li-Wei \cite[Theorem~1.1]{LHW2021} obtained the long-time existence and convergence of flow \eqref{BGL-flow} when   $M_{0}$ is a $h$-convexity for all $1\leq k\le n$, which also provided an alternative proof of the inequalities \eqref{Alex-Fen} for $h$-convex domain $\O\subset \HH^{n+1}$. It is a challenging problem to prove that inequalities \eqref{Alex-Fen} hold for a domain under the weak geometric assumption, say for instance assuming $\p\O$ is $(k-1)$-convex and star-shaped,  which is an analogous known condition to be true for the Alexandrov-Fenchel inequality in Euclidean space (cf. Guan-Li \cite[Theorem 2]{GL2009}). Nevertheless, there have been some efforts and partial results in this direction.  Li-Wei-Xiong \cite[Theorem~1]{LWX} demonstrated that when $k=3$ and $l=1$, \eqref{Alex-Fen} holds for $\p\O$ being $2$-convex and star-shaped.   Andrews-Chen-Wei \cite[Corollary~1.5]{ACW} established \eqref{Alex-Fen} with $k=1, \cdots, n$ and $l=0$ for a domain with boundary having positive intrinsic curvatures, which by Gauss equation is
equivalent to the principal curvatures of $\p\O$ satisfying $\k_i\k_j>1$ for $1\leq i\neq j\leq n$. This is again a weaker condition than $h$-convexity.   Andrews-Hu-Li \cite[Corollary~1.2]{AHL} showed that \eqref{Alex-Fen} holds for a strictly convex domain with $k=n-1$ and $l=n-1-2m $ $(0<2m<n)$. For more related progress in hyperbolic space, one can refer to \cite{AW, BP, BDS, BHW, CM, deG, GWW-2, GWW-1, GL2014, GLW, HL2019, Mak, SX, WX} and references therein.

Meanwhile, there has been growing interest in investigating geometric variational problems for compact hypersurfaces with non-empty boundaries in recent decades, such as hypersurfaces with free or general capillary boundaries in Euclidean space. Especially the studies have focused on isoperimetric type problems (see \cite{BS, BM, CGR, LWW, Maggi} etc.) and Alexandrov-Fenchel type inequalities (see \cite{HWYZ2, SWX, WX2022, WWX1} etc.) for these hypersurfaces in Euclidean space. In particular, Scheuer-Wang-Xia introduced the concept of quermassintegrals for compact hypersurfaces with free boundaries in the Euclidean unit ball $\bar{\BB}^{n+1}$ from a variational perspective in \cite{SWX}. Then they established the Alexandrov-Fenchel inequalities and Gauss-Bonnet-Chern theorem for these quantities, which can be viewed as higher-order generalizations of the relative isoperimetric inequality in $\bar{\BB}^{n+1}$ (cf. \cite[Theorem 18.1.3]{Yu}). They achieved this new family of Alexandrov-Fenchel inequalities for convex hypersurfaces in $\bar{\BB}^{n+1}$ with free boundaries by constructing a locally constrained inverse curvature flow, which is motivated by the Minkowski formula for free boundary hypersurface in \cite[Proposition 5.1]{WX2019}. Subsequently, Weng-Xia \cite{WX2022}  defined the analogous concept of the quermassintegrals for capillary hypersurfaces in $\bar{\BB}^{n+1}$, then they obtained the Alexandrov-Fenchel type inequalities and Gauss-Bonnet-Chern theorem in the capillary setting of $\bar{\BB}^{n+1}$. Very recently, Wang-Weng-Xia \cite{WWX1} introduced the quermassintegrals for compact hypersurfaces with capillary boundary in the Euclidean half-space $\ol{\RR^{n+1}_+}$, and further derived the Alexandrov-Fenchel inequalities for those capillary hypersurfaces. It turns out those new quantities in $\ol{\RR^{n+1}_+}$ can also be interpreted from the viewpoint of convex geometry, as shown in \cite[Section 2.2]{MWWX}.  For more related results, we recommend the readers refer to \cite{HWYZ1, HWYZ2, LS16, LS17, Maggi, MWW, MWWX, MW2023, QWX, WWX1, WWX2, WX20} and references therein.

\ 

Based on the aforementioned results, a natural question arises regarding the corresponding geometric variational problem, specifically the \textit{isoperimetric type problems}, for compact hypersurface with non-empty boundaries in hyperbolic space. This paper's primary objective is to first introduce the quermassintegrals for capillary hypersurfaces supported on totally geodesic hyperplanes in hyperbolic space $\HH^{n+1}$. Subsequently, we establish the Alexandrov-Fenchel inequalities (a higher-order isoperimetric type inequality) for these quantities in $\HH^{n+1}$. To describe our results, we introduce some notations and definitions. We employ the  Poincar\'e ball model to represent the hyperbolic space  $\HH^{n+1}$,  denoted by $\left(\BB^{n+1}, \bar g\right)$ with
	\begin{eqnarray*}
		\BB^{n+1}=\{x\in\RR^{n+1}: |x|<1\}, \quad \bar g= e^{2u}\delta, ~~~~e^{2u}:=\frac{4}{(1-|x|^2)^2},
	\end{eqnarray*}
     where $\delta$ is the standard Euclidean metric. The supported totally geodesic hyperplane is given by   $$\mathcal{H}
 :=\{x\in \BB^{n+1}:\delta(x, E_{n+1})=0\},$$
  where $E_{n+1}=(0,\cdots, 0, 1)$ is the $(n+1)$-th coordinate basis in $\BB^{n+1}$. We further denote
 \begin{eqnarray*}
     \mathcal{H}^{+}:=\{x\in \mathbb{B}^{n+1}:\delta(x, E_{n+1})\geq 0\}.
 \end{eqnarray*}
  Let $\Sigma$ be an embedded hypersurface  in ${\mathcal{H}^{+}}$ satisfying \begin{eqnarray}\label{capi-condition}
      {\rm{int}}(\Sigma)\subset {\rm int}(\H^{+})\quad \text{and}\quad  \partial\Sigma \subset \H.
  \end{eqnarray}
 We denote $\widehat{\Sigma}$ as the bounded domain enclosed by $\Sigma$ and the totally geodesic hyperplane $\H$ in $\H^+$
 and $\widehat{\partial \S}$ as the bounded domain enclosed by $\partial \Sigma$  inside $\H$, see Figure \ref{fig1}. Without loss of generality, throughout this paper, we assume that the origin point $O\in \text{int}(\widehat{\partial\S})$.
 \begin{defn}\label{def1.1}
	A compact hypersurface $\S\subset  {{\H}^{+}}$ is called a capillary hypersurface if 
	it satisfies \eqref{capi-condition} and intersects with $\H$ at a constant contact angle $\theta\in (0,\pi)$ along $\p\S\cap \H$.
\end{defn}
The simplest example of  capillary hypersurface in $\HH^{n+1}$ is a family of geodesic spherical caps lying entirely in ${\H^{+}}$ and intersecting with $\H$ at a constant contact angle $\theta\in (0,\frac{\pi}{2})$, which is given by  
 \begin{eqnarray}\label{static cap model}
		\mathcal{C}_{ \theta,r_{0}}:=\left\{x\in {\H^+}: |x+r_{0}\cos\theta E_{n+1}|= r_{0} \right\},~\text{for}~0<r_{0}<\frac{1}{\sin\theta}.
	\end{eqnarray} It is clear that the constraint $0<r_0< \frac 1{ \sin\theta}$ is a necessary and sufficient condition for $\mathcal{C}_{ \theta,r_{0}}$ lying in the unit ball.

Inspired by the recent progress about the capillary hypersurface in Euclidean space as \cite{SWX, WX2022} and \cite{WWX1}, we first introduce a family of new geometric quantities (Quermassintegrals) $\mathcal{A}_{k,\theta}(\wh\S)$ for capillary hypersurface $\S$ in $\HH^{n+1}$ in this paper.  
Before that, we fix some notations for the $(n+1)$-dimensional convex body $\wh\S\subset \HH^{n+1}$. Parallel to \eqref{quer for closed0} and \eqref{quer for compact hypersurface}, we denote   
\begin{eqnarray}\label{quer for wh Sigma}
\begin{aligned}
    &\mathcal{W}_{0}(\wh\S):=|\wh\S|,\quad ~~ \mathcal{W}_{1}(\wh\S):=\frac{1}{n+1}|\S|,%~~ \W_{n+1}(\wh\S)=\frac{|\SS^n|}{n+1}, 
    \\
    &\mathcal{W}_{k+1}(\wh\S)=\frac{1}{n+1}\int_{\S}H_{k}dA-\frac{k}{n+2-k}\mathcal{W}_{k-1}(\wh\S), ~1\leq k\leq n-1,\notag
    \end{aligned}
    \end{eqnarray} where $H_{k}$ is the normalized $k$-th mean curvature of $\S\subset \mathcal{H}^{+}$. Similarly, we have the quermassintegrals for the $n$-dimensional convex body $\widehat{\partial\Sigma}\subset \H\subset \HH^{n+1}$, which  are defined by 
\begin{eqnarray*}
&&\mathcal{W}_{0}^\H(\widehat{\p\Sigma}):=|\widehat{\p\Sigma}|,\quad \quad \mathcal{W}_{1}^\H(\widehat{\p\Sigma}):=\frac{1}{n}|\p\Sigma|,%~~\W_n^\H (\wh{\p\S})= \frac    {|\SS^{n-1}|}{n}
\\ && \mathcal{W}_{k+1}^\H(\widehat{\p\Sigma})=\frac{1}{n}\int_{\p\Sigma}H_{k}^{\p\S} ds-\frac{k}{n+1-k}\mathcal{W}^\H_{k-1}(\widehat{\p\Sigma}), ~1\leq k\leq n-2,
\end{eqnarray*}where $H_k^{\p\S}:=\frac{1}{\binom{n-1}{k}}  \s^{\p \S}_k $ is the normalized $k$-th mean curvature of $\p\S\subset \H$.

Now we are ready to introduce the quermassintegrals $\mathcal{A}_{k,\theta}(\wh\S)$ for capillary hypersurface $\S$ in $\HH^{n+1}$ as 
\begin{eqnarray*}
    \mathcal{A}_{0, \theta}(\widehat{\Sigma}):=\W_0(\widehat{\Sigma}),~~
    \mathcal{A}_{1, \theta}(\widehat{\Sigma})  
    :=\W_1(\wh\S)-\frac{\cos\theta}{n+1}\W_0^\H(\wh{\p\S}),
    \end{eqnarray*}
   % \begin{eqnarray}        \A_{n+1,\theta}(\wh\S):=\W_{n+1}(\wh\S)+\cdots    \end{eqnarray}
and for $1\leq k\leq n-1$,
    \begin{eqnarray}
         \mathcal{A}_{k+1,\theta}(\widehat{\Sigma})
    :=\mathcal{W}_{k+1}(\widehat{\Sigma})
    +\frac{\cos\theta}{n+1}\sum\limits_{l=0}^{[\frac{k}{2}]}(-1)^{l-1}(\sin\theta)^{k-2l}\cdot\mathcal{W}_{k-2l}^\H (\widehat{\partial\Sigma})
    \prod_{s=0}^{l-1}\frac{k-2s}{n-k+2(s+1)},~~~~~~~\label{A-k,theta}
    \end{eqnarray}
where we have used the convention that $\prod\limits_{s=0}^{-1} \cdot =1$. 

   When $\theta=\frac \pi 2$, it is easy to see that $\A_{k+1,\theta}(\wh\S)= \W_{k+1}(\wh\S)$ for all $-1\leq k\leq n$, which make us to expect that $\mathcal{A}_{k+1, \theta}(\widehat{\S})$ would be the correct capillary counterpart of the quermassintegrals for the closed hypersurfaces in $\HH^{n+1}$. 
Besides, the following first variational formula is the motivation for us to define $\mathcal{A}_{k,\theta}(\wh\S)$  as the quermassintegrals for capillary hypersurface in $\HH^{n+1}$. %with capillary boundary supported on $\H$.
\begin{thm}\label{thm1}
    Let $\Sigma_{t}\subset  {{\H}^{+}}$ be a family of smooth  capillary hypersurfaces,  given by embeddings $x(\cdot, t): M\rightarrow  {{\H}^{+}} $ and satisfying
    \begin{eqnarray*}\label{general flow}
       (\partial_{t} x )^{\perp}=f\nu,
    \end{eqnarray*}
    for some smooth function $f$. Then for $0\leq k\leq n$,  
    \begin{eqnarray*}\label{variation formula}
        \frac{d}{dt}\mathcal{A}_{k,\theta}(\wh{\S_t})=\frac{n+1-k}{n+1}\int_{\Sigma_{t}}f H_{k}dA_{t}.
    \end{eqnarray*} 
  \end{thm}

Furthermore, we establish the Alexandrov-Fenchel inequalities for the quermassintegrals $\A_{k,\theta}(\wh\S)$, when $\S$ is a convex capillary hypersurface in $ \HH^{n+1}$ and $\theta\in(0, \frac \pi 2]$. % concerning $\A_{k,\theta}$. 
\begin{thm}\label{thm AF} 
    For $n\geq 2$, let $\Sigma\subset {{\H}^{+}}$ be a convex capillary hypersurface with  contact  angle $\theta\in (0,\frac{\pi}{2}]$. Assume that
    \begin{eqnarray}\label{add_con}
       \hbox{there exists  a geodesic spherical cap 
         $\C_{\theta, r_0}$ such that $\Sigma\subset\widehat{\C_{\theta, r_0 }}$.~~~~} 
    \end{eqnarray} Then there holds 
    \begin{eqnarray}\label{A-F inqeu}
        \mathcal{A}_{n,\theta}(\widehat{\Sigma})\geq (f_{n,\theta}\circ f_{k,\theta}^{-1})\left(\mathcal{A}_{k,\theta}(\widehat{\Sigma})\right),\quad \forall~ 1~\leq k\leq n-1,
    \end{eqnarray}
 where $f_{k,\theta}:[0,\infty)\rightarrow \RR_{+}$ is a strictly monotone function defined by $$f_{k,\theta}(r):=\mathcal{A}_{k, \theta}(\widehat{\mathcal{C}_{\theta, r}}),$$ 
 where $\mathcal{C}_{\theta, r}$ is the geodesic  spherical cap given by \eqref{static cap model}. Moreover, equality holds if and only if $\Sigma$ is a geodesic spherical cap. 
\end{thm}

In other words, the maximum value of $\A_{k,\theta}(\wh\S)$ among all the convex capillary hypersurfaces $\S\subset \HH^{n+1}$ with a fixed value $\A_{n,\theta}(\wh\S)$ is achieved by the geodesic spherical caps. This solves an isoperimetric type problem for compact hypersurfaces with non-empty boundaries in $\HH^{n+1}$. In particular, when $\theta=\frac \pi 2$, by a simple reflection argument of $\S$ along $\H$, \eqref{A-F inqeu} recovers the Alexandrov-Fenchel inequalities between $\W_k$ and $\W_n$ for convex closed hypersurface in $\HH^{n+1}$, as shown in \cite[Theorem 1.4]{BGL}. 
Furthermore, when $n=2$, from \eqref{A-F inqeu}, we obtain a Minkowski-type inequality for convex capillary surface $\S$ in $\HH^3$.
\begin{cor}\label{cor-Minkowski ineq}
    Let $\Sigma\subset \H^+$ be a convex capillary surface with contact angle $\theta\in(0, \frac \pi 2]$, and $\S$ satisfies \eqref{add_con}. Then
	\begin{eqnarray}\label{minkowski-inequ}
 \int_\S H dA\geq  2|\widehat{\S}|+  6~ (f_{2,\theta}\circ f_{1,\theta}^{-1})\left(\frac{1}{3}\left(|\S|-\cos\theta |\wh{\p\S}|\right)\right) +\sin\theta\cos\theta|\p\S|.
	\end{eqnarray}
Moreover,  equality holds if and only if $\Sigma$ is a geodesic spherical cap. \end{cor}

In order to prove Theorem \ref{thm AF}, we construct a locally constrained inverse curvature flow as described in \eqref{flow with cap-normal} and \eqref{flow with capillary}, inspired by the ideas of Brendle-Guan-Li in \cite{BGL} and the Minkowski formula \eqref{Minkowski formula} by Chen-Pyo \cite{chen} for capillary hypersurfaces in $\HH^{n+1}$. We demonstrate that if the initial capillary hypersurface is strictly convex, the flow exists for all time $t\in[0,\infty)$, preserves the convexity, and smoothly converges to a geodesic spherical cap, as stated in Theorem \ref{flow-result}. A key ingredient to show Theorem \ref{flow-result} is obtaining the uniform curvature estimates, particularly the two-sided uniform bound for $F$. For this, we introduce the  capillary support function $\wt v$ as \eqref{wt v} for capillary hypersurfaces $\S$ in $\HH^{n+1}$, i.e. \begin{eqnarray*}
    \wt v=\frac{\bar g(x,\nu)}{V_0-\cos\theta \bar g(Y_{n+1},\nu)},
\end{eqnarray*} which satisfies a nice homogeneous Neumann boundary value condition along $\p\S$, as shown in \eqref{widetilde v boundary}. Moreover, along the flow \eqref{flow with cap-normal}, we show that quermassintegral $\mathcal{A}_{n,\theta}(\widehat{\Sigma_{t}})$  is preserved, while $\mathcal{A}_{k,\theta}(\widehat{\Sigma_{t}}) $ $(1\leq k\leq n-1)$ is non-decreasing for $t\geq 0$.  This allows us to complete the proof of Theorem \ref{thm AF}, by combining Theorem \ref{flow-result}.  We note that the condition $\theta\leq \frac \pi 2$ is a technical assumption necessary to ensure the boundary curvature estimate of the flow, as seen in \eqref{bdry of H}, which is the only place we used this condition. This angle restriction is similarly required and crucially utilized in \cite{HWYZ1, HWYZ2, WWX1, WX2022} etc. Finally, we point out that the assumption \eqref{add_con} ensures the existence of a geodesic spherical cap that bounds the capillary hypersurface from the exterior, which may not generally be true for convex capillary hypersurfaces  $\S\subset \H^+$ with boundaries  $\p\S$ close to $\H\cap\p \BB^{n+1}$. Given this natural assumption, it is evident that it will be preserved for all evolving convex capillary hypersurfaces by the avoidance principle along the flow \eqref{flow with cap-normal} starting from such initial datum, as stated in Proposition \ref{C0 estimates}.

\

 \textbf{The rest of the article is structured as follows.} In Section \ref{sec2},  we recall some basic properties of elementary symmetric polynomial functions. Subsequently, we introduce relevant notations and basic properties regarding capillary hypersurfaces supported on the geodesic hyperplane $\mathcal{H}$ in $\HH^{n+1}$. Then we present the first variational formula of quermassintegrals $\mathcal{A}_{k,\theta}$ and complete the proof of Theorem \ref{thm1}. In Section \ref{sec3}, we introduce the locally constrained inverse curvature flow \eqref{flow with cap-normal} and analyze the long-time existence and convergence of such flow.  The last section is devoted to proving the Alexandrov-Fenchel inequalities for convex capillary hypersurfaces in  $\HH^{n+1}$, i.e., Theorem \ref{thm AF}.
\vspace{.2cm}

 \section{Quermassintegrals and first variational formula}\label{sec2}
\subsection{Elementary symmetric polynomial functions}\label{sec2.1}
 In this subsection, we recall some well-known properties
of the $k$-th elementary symmetric functions.
   Let $A=\{A_{ij}\}$ be an $n\times n$ symmetric matrix, and $k\in \{1,2,\cdots, n\}$, % is a positive integer, 
    define
    \begin{eqnarray*}\label{k-ele}
        \sigma_{k}(A):=\sigma_{k}(\lambda(A))=\sum\limits_{1\leq i_{1}<i_{2}\cdots< i_{k}\leq n}\lambda_{i_{1}}\lambda_{i_{2}}\cdots \lambda_{i_{k}},
    \end{eqnarray*}
    where $\lambda:=\lambda(A)=(\lambda_{1} ,\cdots,\lambda_{n})$ is the eigenvalues  of $A$.
 We use the convention that $\sigma_0=1$ and $\sigma_k =0$ for $k>n$. Let $H_k(A):=H_k(\lambda(A))$ be the normalization of $\sigma_{k}(\lambda)$ given by $$H_k(\lambda)=\frac{1}{\binom{n}{k}}\s_k(\lambda).$$ Denote  $\sigma _k (\lambda \left| i \right.)$ the symmetric polynomial function of $\s_k(\l)$ with $\lambda_i = 0$ and $\sigma _k (\lambda \left| ij \right.)$ the symmetric polynomial function of $\s_k(\l)$ with $\lambda_i =\lambda_j = 0$.  Recall that G{\aa}rding's cone is defined as
\begin{eqnarray}\label{2.4}
\Gamma_k  = \{ \lambda  \in \mathbb{R}^n :\sigma _i (\lambda ) > 0,~~\forall 1 \le i \le k\}.
\end{eqnarray} 

\begin{lem}\label{prop2.1}Let $\lambda=(\lambda_1,\cdots,\lambda_n)\in\mathbb{R}^n$ and $k
=1, \ldots, n$. Then
	\begin{enumerate}
\item $ \sigma_k(\lambda)=\sigma_k(\lambda|i)+\lambda_i\sigma_{k-1}(\lambda|i),  \quad  \forall 1\leq i \leq n.$
\item $ \sum\limits_{i = 1}^n {\sigma_{k}(\lambda|i)}=(n-k)\sigma_k(\lambda).$
\item $\sum\limits_{i = 1}^n {\lambda_i \sigma_{k-1}(\lambda|i)}=k\sigma_k(\lambda)$.
\item $\sum\limits_{i = 1}^n {\lambda_i^2 \sigma_{k-1}(\lambda|i)}=\s_1(\lambda)\sigma_k(\lambda)-(k+1)\s_{k+1}(\lambda)$.
	\end{enumerate}
\end{lem}

We denote $\sigma _k(A \left|
i \right.)$ the $\sigma_k$ symmetric polynomial function of the matrix 
obtained from $A$ by deleting the $i$-row and
$i$-column and $\sigma _k (A \left| ij \right.)$ the $\s_k$ symmetric polynomial 
function of the matrix from $A$ by deleting the $i,j$-rows and $i,j$-columns. 
\begin{lem}%\label{prop2.2}
Suppose that  $A=\{A_{ij}\}$ is diagonal, and $k$ is a positive integer.
Then %we have
\begin{eqnarray*}
\sigma_{k-1}^{ij}(A)= \begin{cases}
\sigma _{k- 1} (A\left| i \right.), &\text{if } i = j, \\
0, &\text{if } i \ne j,
\end{cases}
\end{eqnarray*}
where $\sigma_{k-1}^{ij}(A):=\frac{{\partial \sigma _k (A)}} {{\partial A_{ij} }}$.
\end{lem}
\begin{lem}\label{pro2.3}
    The following properties hold.
    \begin{enumerate}
        \item  For $\lambda \in \Gamma_k$ and $k > l \geq 0$, $ r > s \geq 0$, $k \geq r$, $l \geq s$, there holds  the generalized Newton-Maclaurin inequality
\begin{eqnarray*} \label{1.2.6}
\left(\frac{H_{k}(\lambda)}{H_{l}(\lambda)}\right)^{\frac{1}{k-l}}\leq \left(\frac{H_{r}(\lambda)}{H_{s}(\lambda)}\right)^{\frac{1}{r-s}},
\end{eqnarray*}
with equality holds if and only if $\lambda_1 = \lambda_2 = \cdots =\lambda_n >0$.
\item For $0\le l<k\le n$, then $\left(\frac{H_k(\lambda)}{H_{l}(\lambda)}\right)^{\frac{1}{k-l}}$ is a concave function with respect to  $\l\in \Gamma_
k$.
\item For $0\leq l<k\leq n$, and  $F(\lambda)=\left(\frac{H_{k}(\lambda)}{H_{l}(\lambda)}\right)^{\frac{1}{k-l}}$. Assume $\lambda=(\lambda_{1},\cdots, \lambda_{n})$ satisfies $\lambda_{1}\geq \lambda_{2}\geq\cdots\geq \lambda_{n}$. Then 
    \begin{eqnarray*}
      \frac{\partial F(\lambda)}{\partial \lambda_{1}}\leq \frac{\partial F(\lambda)}{\partial\lambda_{2}}\leq \cdots\leq \frac{\partial F(\lambda)}{\partial\lambda_{n}}.
\end{eqnarray*}
\end{enumerate}
\end{lem}
For proof of the above Lemmas, see e.g. \cite[Chapter \uppercase\expandafter{\romannumeral 15}, Section~4]{Lie}  and  \cite[Lemma~2.10, Theorem~2.11, Lemma~1.5]{Spruck}.

\subsection{Notation and conventions}
We use $D$ to denote the Levi-Civita connection of $\H^{+}$ w.r.t the metric $\bar{g}$, 
 and $\n$  represents the Levi-Civita connection on $\S$ w.r.t the induced metric $g$ from the immersion $x$. 
	The operators $\div, \Delta$, and $\n^2$ are the divergence, Laplacian, and Hessian operators on $\S$ respectively.
	The second fundamental form $h$ of $x$  is defined by
	$$D_{X}Y=\n_{X}Y- h(X,Y)\nu.$$
	%Denote $\k=(\k_1, \k_2,\cdots, \k_n)$ be the set of principal curvatures, i.e, the set of eigenvalues of $h$. 
The Weingarten operator is defined via
$\bar g(\W(X),Y)=h(X,Y),$
and the Weingarten equation is $$D_{X}\nu=\W(X).$$
 We shall use the convention of Einstein summation. For convenience the components of the Weingarten map $\W$ are denoted by $(h^{i}_{j})=(g^{ik}h_{kj})$, and $|h|^2$ be the norm square of the second fundamental form, that is $|h|^2=g^{ik}h_{kl}h_{ij}g^{jl}$, where  $(g^{ij})$ is the inverse of $(g_{ij})$. We use the metric tensor $(g_{ij})$ and its inverse $(g^{ij})$ to lower down and raise up the indices of tensor fields on $\S$. 

  \subsection{Convex capillary hypersurfaces in $\HH^{n+1}$}
     Let  $\S\subset {{\H}^{+}}$ be a smooth capillary hypersurface, given by the embedding $x: M\to  {{\H}^{+}}$,      if without cause confusion, we do not distinguish $\Sigma$  and the embedding $x$.
   Let $\mu$ be the unit outward co-normal of $\p\S$ in $\S$ and  $\ov{\nu}$ be the unit normal to $\partial\Sigma$ in $\H$ such that $\{\nu,\mu\}$ and $\{\ov{\nu},\ov{N}\}$ have the same orientation in normal bundle of $\partial\Sigma\subset  {{\H}^{+}}$, where $\ov{N}$ is the unit outward normal of $\H\subset {\H}^{+}$. See Figure \ref{fig1}. From Definition \ref{def1.1},  
\begin{eqnarray}\label{capi-boundary}
  \bar{g}( \nu, \ov{N})=\cos (\pi-\theta).  
\end{eqnarray}
It  follows  
\begin{eqnarray}\label{co-normal bundle}
	\begin{array} {rcl}
		\ov{N} &=&\sin\theta \mu-\cos\theta \nu,
		\\
		\ov{\nu} &=&\cos\theta \mu+\sin\theta \nu.
	\end{array}
\end{eqnarray}

\begin{center}	\begin{figure}[H] \includegraphics[width=0.45\linewidth]{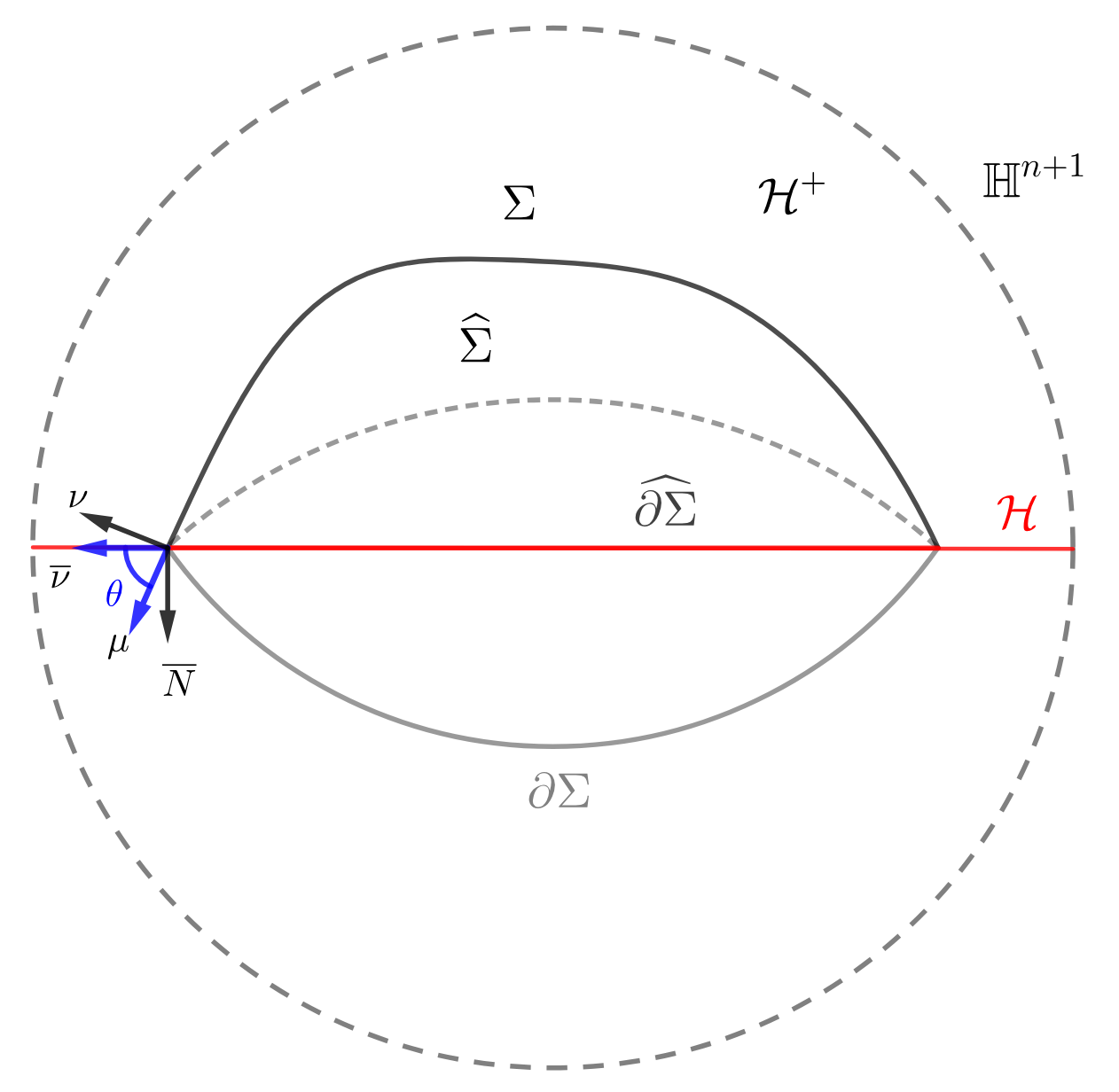}  \caption{Figure 1. A capillary hypersurface $\S\subset\HH^{n+1}$ supported on $\H$.}   \label{fig1}	\end{figure} \end{center}
The second fundamental form of $\p \S$ in  $\mathcal H$ is given by $$\widehat{h}(X, Y):= \bar{g}(\nabla^{\HH^n}_X \ov{\nu}, Y)= \bar{g}(D_X \ov{\nu}, Y), \quad X, Y\in T(\p\S).$$
	The second fundamental form of $\p \S$ in $\S$ is given by
	$$\widetilde{h}(X, Y):= \bar{g}(\nabla_X \mu, Y)= \bar{g}(D_X \mu, Y), \quad X, Y\in T(\p\S).$$
It turns out that  $\widehat{h}, \widetilde{h}$  and $h$ have a nice relationship. The similar properties were previously established by Wang-Weng-Xia  in \cite[Proposition 2.2]{WWX1} for the capillary hypersurface in $\ol{\RR^{n+1}_{+}}$.

\begin{prop}\label{basic-capillary}
Let $\S\subset  {\H^+}$ be a smooth capillary hypersurface and $\{e_\a\}_{\a=2}^{n }$ be an orthonormal frame of $\p \S$. Then along $\p\S$,	\begin{itemize}		\item[(1)] $\mu$ is a principal direction of $\S$. That is, $h_{\mu \a}:=h(\mu, e_\a)=0$.		\item[(2)] $h_{\a\b}=\sin\theta \widehat{h}_{\a\b} .$		\item[(3)] $\widetilde{h}_{\a\b}=\cos\theta \widehat{h}_{\a\b} =  \cot\theta h_{\a\b}.$		\item[(4)] $\n_\mu h_{\a\b}=\widetilde{h}_{\b\g}(h_{\mu\mu}\delta_{\a\g}-h_{\a\g})$.	\end{itemize}\end{prop}
\begin{proof}
    The first and fourth assertions are well-known, see e.g. \cite[Lemma 2.2]{AiSo} and \cite[Proposition~2.4 (4)]{WWX1} resp. While (2) and (3) follow
    from  \eqref{co-normal bundle} and the simple fact that $\H$ is totally geodesic hyperplane in $\HH^{n+1}$. In fact, 
     \begin{eqnarray*}
        \widehat{h}_{\alpha\beta}=\bar{g}(D_{e_{\alpha}}\ov{\nu}, e_{\beta})=\bar{g}(D_{e_{\alpha}}(\cot\theta\ov{N}+\csc\theta\nu), e_{\beta})=\csc\theta h_{\alpha\beta},
    \end{eqnarray*}
    and 
    \begin{eqnarray*}
        \widetilde{h}_{\alpha\beta}=\bar{g}(D_{e_{\alpha}}\mu, e_{\beta} )=\bar{g}(D_{e_{\alpha}}(\csc\theta \ov{N}+\cot\theta \nu), e_{\beta})=\cot\theta h_{\alpha\beta}.
    \end{eqnarray*}
\end{proof}

\begin{cor}\label{coro convex}
    Let $\Sigma\subset {\H^+}$ be a  smooth capillary hypersurface  %supported on  geodesic plane $\H$ 
    with a constant angle  $\theta\in (0,\frac{\pi}{2})$. If $\Sigma$ is convex (resp. strictly convex), in the sense that $h$ is non-negative definite (resp. positive definite), then $\partial\Sigma$ is convex (resp. strictly convex) in both  $\Sigma$ and $\mathcal H$. 
\end{cor}
Chen-Pyo had established the following Minkowski type formula \cite[Proposition~3]{chen},  which can be viewed as the capillary analogous result of \eqref{mink formula-closed} in $\HH^{n+1}$. This formula will be important in constructing the locally constrained curvature flow later.  
 \begin{prop}[\cite{chen}] \label{Minkowski}
       Let $x: M\rightarrow  {{\H}^{+}}$ be a smooth immersion of $\Sigma:=x(M)$ into $ {{\H}^{+}}$, and its boundary intersects with $\H$ at a constant angle $\theta\in (0,\pi)$. Then for any $1\leq k\leq n$, there holds  
     \begin{eqnarray}\label{Minkowski formula}
         \int_{\Sigma}\left[H_{k-1}(V_{0}-\cos\theta\bar{g}(Y_{n+1},\nu))-H_{k}\bar{g}(x, \nu)\right]dA=0,
     \end{eqnarray}
     where  $V_{0}=\frac{1+|x|^{2}}{1-|x|^{2}}$, $Y_{n+1}$ is a Killing vector field in $\HH^{n+1}$, 
     given by
     \begin{eqnarray}\label{Y}
         Y_{n+1}:=\frac{1}{2}(1+|x|^{2})E_{n+1}-\delta(x, E_{n+1})x,
     \end{eqnarray}
     and $dA$ is the area element of $\S$ w.r.t. the induced metric.
 \end{prop}

 Next, we demonstrate that the spherical cap $\C_{\theta, r_{0}}$ is an umbilical capillary hypersurface in $\H^+$ such that the integrand in \eqref{Minkowski} is identically zero.  In particular, this implies that the geodesic spherical caps are the static solutions to our flow  \eqref{flow with capillary} in the next Section.

 \begin{prop}\label{sup of model}
 For any $r_0>0$, the geodesic spherical cap  $\C_{\theta, r_{0}}$ defined in \eqref{static cap model} satisfies 
    \begin{eqnarray}\label{static solution}
        V_{0}-\cos\theta\bar{g}(Y_{n+1},\nu)-\frac{1+r^{2}_{0}\sin^{2}\theta }{2r_{0}}\bar{g}(x,\nu)=0,
    \end{eqnarray}
   and its principal curvatures equal  
$\frac{H_k}{H_{k-1}}=\frac{1+r^{2}_{0}\sin^{2}\theta }{2r_{0}}$ for all $1\leq k\leq n$.
\end{prop}
\begin{proof}
    For any $x\in \C_{\theta, r_{0}}$, we know  %direct calculations yield
    \begin{eqnarray}
    \label{Yn+1 in spherical caps}
 \cos\theta    \bar{g}(Y_{n+1}, \nu)&=&\frac{4  \cos\theta }{(1-|x|^{2})^{2}}\left(\frac{1}{2}(1+|x|^{2})\delta(E_{n+1},\nu)-\delta(x, E_{n+1})\delta(x, \nu)\right)\notag \\
    &=&\frac{r^{2}_{0}\cos^{2}\theta+r^{2}_{0}-|x|^{2}+r^{2}_{0}|x|^{2}\cos^{2}\theta+r^{2}_{0}|x|^{2}-r^{4}_{0}\sin^{4}\theta}{2(1-|x|^{2})r^{2}_{0}},
     \end{eqnarray}
     and 
     \begin{eqnarray}\label{support function of spherical caps}
         \bar{g}(x,\nu)=\frac{4}{(1-|x|^{2})^{2}}\delta(x, \nu)=\frac{|x|^{2}+r^{2}_{0}\sin^{2}\theta}{r_{0}(1-|x|^{2})}.
         \end{eqnarray}
         Combining \eqref{Yn+1 in spherical caps}, \eqref{support function of spherical caps} and $V_{0}=\frac{1+|x|^{2}}{1-|x|^{2}}$,  it follows 
         \eqref{static solution}. Using \cite[formula~(13)]{WX20}, 
        the spherical cap $\C_{\theta, r_{0}}$ is umbilical in ${\H}^{+}$ and the principal curvatures are equal to $\frac{1+r^{2}_{0}\sin^{2}\theta}{2r_{0}}$, then the assertion is proved. 
\end{proof}

 \subsection{Quermassintegrals and the first variational formula}  
 \label{sec2.4}
We introduce the quermassintegrals $\A_{k,\theta}(\wh\S)$ as \eqref{A-k,theta} for compact hypersurface with capillary boundary supported on the geodesic hyperplane in $\HH^{n+1}$. The definition of the quermassintegrals $\mathcal{A}_{k, \theta}(\widehat{\S})$ in such a way is motivated by the first variational formula. For the reader's convenience, we restate Theorem \ref{thm1} in the following.
\begin{thm}\label{variation thm}
    Let $\Sigma_{t}\subset {\H^+}$ be a family of smooth hypersurfaces with capillary boundary supported on $\H$,  given by the embedding $x(\cdot, t): M\rightarrow  {\H^+}$ and satisfies
    \begin{eqnarray}\label{general flow}
       (\partial_{t} x )^{\perp}=f\nu,
    \end{eqnarray}
    for some smooth function $f$. Then for $-1\leq k\leq n-1 $, 
    \begin{eqnarray}\label{variation formula}
        \frac{d}{dt}\left(\mathcal{A}_{k+1,\theta}(\widehat{\Sigma_{t}})\right)=\frac{n-k}{n+1}\int_{\Sigma_{t}}f H_{k+1}dA_{t}.
    \end{eqnarray} 
\end{thm}
Before proving Theorem \ref{variation thm}, we rewrite the flow \eqref{general flow} in the  following general form 
\begin{eqnarray}\label{more general flow}
    \partial_{t}x=f\nu+\T,
\end{eqnarray}
where $\T$ is the tangential vector along $T\S_t$. If further imposing the capillary boundary condition, then it must satisfy
\begin{eqnarray}\label{T bdry}
   \T |_{\p M}=f\cot\theta \mu,%, \quad \hbox{ on } \p\S.
\end{eqnarray} see e.g. \cite[Section 2.4]{WX2022} or  \cite[Section 2.5]{WWX1}. 

Along the general flow \eqref{more general flow}, we have the evolution equations for the induced metric $g_{ij}$, the volume element $dA_t$, the unit normal vector field $\nu$, the second fundamental form $(h_{ij})$, the Weingarten	tensor  $\W:=(g^{ik}h_{kj})$, the mean curvature $H$ and the Weingarten curvature function $F:=F(\W)$ of hypersurfaces $\Sigma_t$ in $\HH^{n+1}$.  
 \begin{prop}\label{basic evolution eq}
		Along the general flow \eqref{more general flow}, there holds
		\begin{enumerate} 
			\item $\p_t g_{ij}=2fh_{ij}+\n_i \T_j+\n_j\T_i$.
			\item $\p_tdA_t =\left(fH+\div(\T)\right)dA_t.$ 
			\item $\p_t\nu =-\n f+h(e_i,\T)e_i$.	
			\item $\p_t h_{ij}=-\n^2_{ij}f +fh_{ik}h_{j}^k+fg_{ij} +\n_\T h_{ij}+h_{j}^k\n_i\T_k+h_{i}^k\n_j \T_k.$
			\item $\p_t h^i_j=-\n^i\n_{j}f -fh_{j}^kh^{i}_k+f g_{j}^{i}+\n_\T h^i_j.$
			\item $\p_t H=-\Delta f-|h|^2 f+ nf+\n_\T H$.
                \item $\partial_{t}F=-F_{i}^{j}\n^{i}\n_{j}f-f F_{i}^{j}h_{j}^{k}h^{i}_{k}+f F_{i}^{j}g_{j}^{i}+\<\n F, \T\>$, where $F_{j}^{i}:=\frac{\partial F }{\partial h_{i}^{j}}$.
		\end{enumerate}
	\end{prop}
 \begin{proof}
   The first three equations follow a similar computation as in \cite[Lemma 3.3]{Hui86} or \cite[Proposition~2.11]{WX2022}. For (4), it can be derived by combining with \cite[Proposition~2.11]{WX2022} and    \cite[Theorem 3.4]{Hui86}, just noticing now the ambient space has negative constant sectional curvature $-1$. Using (4), then (5)-(7) can be derived directly using the same argument as in \cite[Proposition 2.11 (5)-(7)]{WX2022} respectively.
     \end{proof}
 
 Now we are ready to prove Theorem \ref{variation thm}.
\begin{proof}[\textbf{Proof  of Theorem \ref{variation thm}.}]
Choose an orthonormal frame $\{e_{\alpha}\}_{\alpha=2}^{n}$ of $T\partial\Sigma$ such that $\{e_{1}=\mu, (e_{\alpha})_{\alpha=2}^{n}\}$ forms an orthonormal frame for $T\Sigma$. Taking the time derivative of capillary boundary condition \eqref{capi-boundary}, using Proposition \ref{basic evolution eq} and \eqref{co-normal bundle}, we have
 \begin{eqnarray*}
     0&=&\bar{g}\left(\partial_{t}\nu, \ov{N}(x(\cdot, t))\right)+\bar{g}\left(\nu(x(\cdot, t)), d\ov{N}(f\nu+\T)\right)\\
     &=&-\sin\theta\n_{\mu}f+\sin\theta h(e_{i}, \T)\bar{g}(e_{i},\mu)\\
     &=&-\sin\theta\n_{\mu}f+\sin\theta h(\mu,\mu)\cot\theta f,
 \end{eqnarray*}
 hence
 \begin{eqnarray}\label{boundary condition}
     \n_{\mu}f=\cot\theta h(\mu, \mu)f \quad {\rm{on}}\quad \partial\Sigma_{t}.
 \end{eqnarray}
   Flow \eqref{general flow} induces a hypersurface flow along $\partial \Sigma_{t}\subset \H$ with the normal speed $\frac{f}{\sin\theta}$, i.e.,
    \begin{eqnarray}   \label{flow-partial Sigma}   \partial_{t}x|_{\partial\Sigma_{t}}=f\nu+f\cot\theta\mu=\frac{f}{\sin\theta}\ov{\nu}.
    \end{eqnarray}
    Applying \eqref{variation of quer} to $\wh{\p\S_t}\subset \H$ along the flow \eqref{flow-partial Sigma}, for any $0\leq k\leq n-1$, %(see e.g. \cite[Proposition 3.1]{WX2014})
    \begin{eqnarray}\label{bry variation}\label{bry variation}
        \frac{d}{dt}\mathcal{W}_{k}^\H (\widehat{\partial\Sigma_{t}})=\frac{n-k}{n}\int_{\partial\Sigma_{t}}H_{k}^{\partial\Sigma_{t}}\frac{f}{\sin\theta}ds_{t},
    \end{eqnarray}
    where $H_{k}^{\partial\Sigma_{t}}:=\frac 1 {\binom{n-1}{k} } \s_k^{\p\S_t}$ is the normalized $k$-th mean curvature  of $\partial\Sigma_{t}\subset \H$, $ds_{t}$ is the area element of $\partial\S_{t}$ w.r.t the induced metric.

    First, we show that \eqref{variation formula} for the case $k=-1$ and $k=0$.   
    From direct computations,
 we see    \begin{eqnarray*}\label{variation of vol}
        \frac{d}{dt}\mathcal{A}_{0, \theta}(\widehat{\Sigma_{t}}) %=\frac{d}{dt}(|\widehat{\Sigma_{t}}|)
        =\int_{\Sigma_{t}}f dA_{t}.
    \end{eqnarray*}
By  Proposition \ref{basic evolution eq} (2) and \eqref{T bdry}, %$\T|_{\partial \Sigma_{t}}=f\cot\theta \mu$, we have
    \begin{eqnarray*}\label{variation of area} 
    \begin{aligned}
    \frac{d}{dt}\left(\mathcal{A}_{1,\theta}(\widehat{\Sigma_{t}})\right)
        =&\frac{1}{n+1}\int_{\Sigma_{t}}\left(fH+{\rm{div}}(\T)\right)dA_{t}-\frac{\cot\theta}{n+1}\int_{\partial \Sigma_{t}}f ds_{t}\\
        =&\frac{n}{n+1}\int_{\Sigma_{t}}f H_{1}dA_{t}+\frac{1}{n+1}\int_{\partial \Sigma_{t}}\bar{g}(\T, \mu)ds_{t}-\frac{\cot\theta}{n+1}\int_{\partial \Sigma_{t}}f ds_{t}\\
        =&\frac{n}{n+1}\int_{\Sigma_{t}}f H_{1}dA_{t}.
        \end{aligned}
    \end{eqnarray*}
In order to prove the case $1\leq k \leq n-1$ in \eqref{variation formula}, we first \textbf{claim} that: 
when $k$ is even 
    \begin{eqnarray}\label{even case}
    \begin{aligned}
       & \frac{d}{dt}\mathcal{W}_{k+1}(\widehat{\Sigma_{t}})\\
       =&\frac{n+1-(k+1)}{n+1}\int_{\Sigma_{t}}f H_{k+1}d A_{t}+\frac{\cos\theta}{n+1}(\sin\theta)^{k}\frac{d}{dt}\mathcal{W}_{k}^\H(\widehat{\partial\S_{t}})\\
        &-\frac{\cos\theta}{n+1}\sum\limits_{l=0}^{\frac{k}{2}-1}\left((-1)^{l}(\sin\theta)^{k-2l-2}\frac{d}{dt}\mathcal{W}^\H_{k-2l-2}(\widehat{\partial\S_{t}})\prod_{s=0}^{l}\frac{k-2s}{n-k+2(s+1)}\right),
        \end{aligned}
    \end{eqnarray}
     and when $k$ is odd
    \begin{eqnarray}\label{odd case}
      \begin{aligned}
        &\frac{d}{dt}\mathcal{W}_{k+1}(\widehat{\Sigma_{t}})\\
        =&\frac{n+1-(k+1)}{n+1}\int_{\Sigma_{t}}f H_{k+1}dA_{t}+\frac{\cos\theta}{n+1}(\sin\theta)^{k}\frac{d}{dt}\mathcal{W}_{k}^\H(\widehat{\partial\S_{t}})\\
        &-\frac{\cos\theta}{n+1}\sum\limits_{l=0}^{\frac{k-1}{2}-1}\left((-1)^{l}(\sin\theta)^{k-2l-2}\frac{d}{dt}\mathcal{W}^\H_{k-2l-2}(\widehat{\partial\S_{t}})\prod_{s=0}^{l}\frac{k-2s}{n-k+2(s+1)}\right).
        \end{aligned}
    \end{eqnarray}
We prove the \textbf{claim} using the induction argument. % In view of \eqref{variation of vol} and \eqref{variation of area}, 
    Assume %it is true for $k=-1, 0$. Assume 
   that $k$ is even and \eqref{even case} is true for $k$,
    using Lemma \ref{prop2.1}, Proposition \ref{basic evolution eq} (2) (5), Proposition \ref{basic-capillary} (1) (2), the divergence theorem and  \eqref{boundary condition},  we derive 
    \begin{eqnarray*}
       && \frac{d}{dt}\int_{\S_{t}}H_{k+2}dA_{t}\\&=&\int_{\S_{t}}\frac{\partial H_{k+2}}{\partial h_{i}^{j}} \left(-f^{i}_{j}-fh_{j}^{k}h_{k}^{i}+f\bar{g}_{j}^{i}+\bar{g}(\n h_{j}^{i}, \T)\right)dA_{t}+\int_{\S_{t}}H_{k+2}(nfH_1+{\rm{div}}(\mathcal{T}))dA_{t}\\
     %   &=&(k+3)\int_{\S_{t}}f\sigma_{k+3}d\mu_{t}+(n-k-1)\int_{\S_{t}}f\sigma_{k+1}d\mu_{t}        -\int_{\partial\S_{t}}\frac{\partial\sigma_{k+2}}{\partial h_{i}^{j}}f^{i}\mu_{j}ds_{t}\\        &&+\int_{\partial\S_{t}}\sigma_{k+2}\bar{g}(\T,\mu)ds_{t}\\
        &=&(n-k-2)\int_{\S_{t}}fH_{k+3}dA_{t}+(k+2) \int_{\S_{t}}f H_{k+1}dA_{t}
        +\frac{n-k-2} n \int_{\partial\S_{t}}\cos\theta \sin^{k+1}\theta fH_{k+2}^{\p \S_t}  ds_{t}.
    \end{eqnarray*}
Together with \eqref{bry variation} and \eqref{even case},  it follows
\begin{eqnarray*}
    \frac{d}{dt}\mathcal{W}_{k+3}(\widehat{\S_{t}})&=&\frac{1}{n+1} \frac{d}{dt}\left(\int_{\S_{t}}H_{k+2}dA_{t}\right)-\frac{k+2}{n-k}\frac{d}{dt}\mathcal{W}_{k+1}(\widehat{\S_{t}})\\
    &=&\frac{n+1-(k+3)}{n+1}\int_{\S_{t}}fH_{k+3}dA_{t}+\frac{k+2}{n+1}\int_{\S_{t}}fH_{k+1}dA_{t}\\
    &&+\frac{\cos\theta}{n+1}(\sin\theta)^{k+2} \frac{d}{dt}\mathcal{W}_{k+2}^\H(\widehat{\partial\S_{t}})-\frac{k+2}{n-k}\frac{d}{dt}\mathcal{W}_{k+1}(\widehat{\S_{t}})\\
    %-\frac{k+2}{n-k}\Big(\frac{n-(k+1)}{n+1}\int_{\S_{t}}f H_{k+1}d\mu_{t}+\frac{1}{n+1}(\sin\theta)^{k}\cos\theta\frac{d}{dt}\mathcal{W}_{k}(\widehat{\partial\S_{t}})
    %\Big)
    %\\
    &=&\frac{n+1-(k+3)}{n+1}\int_{\S_{t}}f H_{k+3}dA_{t}+\frac{\cos\theta}{n+1}(\sin\theta)^{k+2}\frac{d}{dt}\mathcal{W}_{k+2}^\H(\widehat{\partial\S_{t}})\\
    &&-\frac{\cos\theta}{n+1}\sum\limits_{l=0}^{\frac{k}{2}}\left((-1)^{l}(\sin\theta)^{k-2l}\frac{d}{dt}\mathcal{W}_{k-2l}^\H(\widehat{\partial\S_{t}})\prod_{s=0}^{l}\frac{k+2-2s}{n-k+2s}\right),
\end{eqnarray*}where we have used the inductive assumption \eqref{even case} of $k$ in the last equality. Therefore, \textbf{claim} \eqref{even case} is proved,  and using a similar argument yields \eqref{odd case} when $k$ is odd. Combining \eqref{bry variation}, \eqref{even case} and \eqref{odd case}, we conclude that \eqref{variation formula} holds and completes the proof of Theorem \ref{variation thm}. 
        \end{proof}

        \section{Locally constrained inverse curvature flow}\label{sec3}
    In this section, we introduce a locally constrained inverse curvature flow inspired by the idea in \cite{BGL}, see also \cite{GL2014, GLW} etc. Let $\S_0\subset {\H}^{+}$ be a smooth,  strictly convex capillary hypersurface, given by the embedding $x_0: M\to \mathbb{\H}^{+}$, and $\S_t\subset \H^{+}$ be a family of smooth,  strictly convex capillary hypersurfaces, given by $x(\cdot, t): M\rightarrow \H^{+}$ starting from $x(M,0)=\S_0$ and satisfying  
    	\begin{equation}\label{flow with cap-normal}
		\begin{array}{rlll}
	 \partial_t x(\cdot, t)  &=& \frac{1}{\sin^2\theta} f(\cdot, t) \wt \nu(\cdot, t), \quad &
			\hbox{ in }M\times[0,T),\\
			\bar{g}\left( \wt \nu(\cdot, t),\ov{N}\circ x(\cdot,t)\right)  &=& 0, 
			\quad & \hbox{ on }\partial M\times [0,T),
		\end{array}
	\end{equation}
where we denote \begin{eqnarray}\label{capillary normal nu}
    \wt\nu:=\nu- \frac{\cos\theta}{V_0} Y_{n+1},\end{eqnarray} and call it \textit{the capillary outward normal} of $\S$ in $\HH^{n+1}$. The choice of $\wt\nu$ ensures that the boundary of $x(M,t)$  evolves inside $\H$ along the flow \eqref{flow with cap-normal}, as indicated by the boundary condition, see also Eq. \eqref{wt nu bdry}.
We remark that $\wt \nu$ is not a unit vector field as the usual normal vector field $\nu$ when $\theta\neq \frac \pi 2$. From Eqs. \eqref{co-normal bundle} and \eqref{Y}, we see
\begin{eqnarray*}
    Y_{n+1}|_{\p M}=V_0(\cos\theta \nu-\sin\theta \mu),
\end{eqnarray*}it follows
\begin{eqnarray}\label{wt nu bdry}
    \wt\nu|_{\p M}= \sin\theta(\sin\theta \nu+\cos\theta \mu)=\sin\theta \ol\nu,
\end{eqnarray}then the boundary equation in \eqref{flow with cap-normal} is equivalent to the capillary boundary condition \eqref{capi-boundary}, as stated  \begin{eqnarray*}    \bar g(\nu, \ol N\circ x)=-\cos\theta,~~~~~~~~~~~\p M.\end{eqnarray*} 
In other words, up to a tangential diffeomorphism on $T\S_t$, \eqref{flow with cap-normal} is equivalent to the following flow
	\begin{equation}\label{flow with capillary}
		\begin{array}{rlll}
		\left(\partial_t x(\cdot, t)\right)^{\perp}  &=& f(\cdot, t) \nu(\cdot, t), \quad &
			\hbox{ in }M\times[0,T),\\
			\bar{g}\left( \nu(\cdot, t),\ov{N}\circ x(\cdot,t)\right)  &=&  -\cos\theta, 
			\quad & \hbox{ on }\partial M\times [0,T).
		\end{array}
	\end{equation} 
 From now on, we specifically define the speed function $f$ in  \eqref{flow with cap-normal} and \eqref{flow with capillary} as follows: \begin{eqnarray}\label{speed f}
		f:=\frac{V_{0}-\cos\theta\bar{g}(Y_{n+1},\nu)}{F}-\bar{g}(x,\nu),
	\end{eqnarray} and  \begin{eqnarray}\label{F=high order}F:=\frac{H_n}{H_{n-1}}.%, ~~~v:=\bar g(\phi(\rho)\p_\rho,\nu).  %=\frac{n\sigma_n }{\sigma_{n-1}}.
	\end{eqnarray}    When $\theta=\frac \pi 2$, the flows \eqref{flow with cap-normal} and \eqref{flow with capillary} reduce to the flow \eqref{BGL-flow}, which was introduced by Brendle-Guan-Li in  \cite{BGL} for closed hypersurfaces in $\HH^{n+1}$.  Along our flow \eqref{flow with cap-normal} or \eqref{flow with capillary}, we have the following monotone property for the quermassintegrals $\A_{k,\theta}$, which is crucial for us to show Theorem \ref{thm AF}.
         \begin{prop}\label{monotone along flow}
		Along the flow \eqref{flow with capillary}, 
		$ \mathcal{A}_{n,\theta}(\widehat{\Sigma}_{t})$ is preserved and $ \mathcal{A}_{k,\theta}(\widehat{\Sigma}_{t})$ is non-decreasing with respect to the time $t>0$ for $1\leq k\leq n-1$.
	\end{prop}
	
	\begin{proof}
	 From  Proposition  \ref{Minkowski} and Theorem \ref{variation thm}, we have 
		\begin{eqnarray*}
			\frac{d}{dt}  \mathcal{A}_{n,\theta}(\widehat{\S_t})=\frac{1}{n+1}\int_{\S_t}  \left[H_{n-1}\left(V_{0}-\cos\theta\bar{g}(Y_{n+1},\nu) \right)-H_n\bar{g}(x,\nu)\right]dA_{t}  = 0.
		\end{eqnarray*}
Using Proposition  \ref{Minkowski} and Theorem \ref{variation thm} again, together with Lemma \ref{pro2.3} (1), for $1\leq k\leq n-1$, we derive
		\begin{eqnarray*}
			\frac{d}{dt}   \mathcal{A}_{k,\theta}(\widehat{\S_t})
			&=&\frac{n+1-k}{n+1}\int_{\S_t} H_{k}	\left[
			\frac{H_{n}}{H_{n-1}}\left(V_{0}-\cos\theta\bar{g}(Y_{n+1},\nu) \right)-\bar{g}(x,\nu) \right]dA_{t}
			\\&\geq &\frac{n+1-k}{n+1}\int_{\S_t}\left[ H_{k-1} \left(V_{0}-\cos\theta\bar{g}(Y_{n+1},\nu) \right)- H_{k} \bar{g}(x,\nu)\right]dA_t \\&= &0.
		\end{eqnarray*}
 \end{proof}
  The primary objective of this section is to establish the long-time existence and convergence of flow \eqref{flow with capillary}.
  	\begin{thm}\label{flow-result}
		Let $\S_0\subset {\H^+}$ be a smooth, strictly convex capillary hypersurface with constant angle $\theta\in (0,\frac{\pi}{2}]$, given by the embedding $x_0: M\to {\H^+}$. If  $\Sigma_{0}$ satisfies \eqref{add_con} and the origin point lies in the interior of $\widehat{\partial \S_{0}}$, then the solution $x(\cdot, t)$ to flow \eqref{flow with capillary} exists for all time $t\in [0,\infty)$.  Moreover, $x(\cdot, t)$ converges smoothly to a geodesic spherical cap $\C_{\theta, r_{*}}$ for some $r_*>0$, where $r_*$ is uniquely determined by the identity $\A_{n,\theta}(\wh{\C_{\theta,r_*}})=\mathcal{A}_{n,\theta}(\wh{\S_0})$.  
	\end{thm}

\subsection{Scalar equation}
In this subsection, we reduce the flow \eqref{flow with capillary} to a scalar parabolic equation with oblique boundary value condition on $\bar{\SS}^{n}_{+}$, if the evolving hypersurfaces are star-shaped.  
Under  the polar coordinate $(r,\zeta, \theta)\in [0, 1)\times[0,\frac{\pi}{2}]\times\SS^{n-1}$,  the standard Euclidean metric in $ \BB^{n+1}_+$ has the following form 
\begin{eqnarray*}|dz|^{2}=dr^{2}+r^{2} \sigma %\bar{g}_{\bar{\SS}^{n}_{+}}
=dr^{2}+r^{2} \left(d\zeta^{2}+\sin^{2}\zeta g_{\SS^{n-1}}\right),
\end{eqnarray*} where $\sigma$ is the standard spherical metric on $\SS^n$, 
then the constant vector field $E_{n+1}$ is given by
\begin{eqnarray*}
   E_{n+1}=\cos\zeta \partial_{r}-\frac{\sin\zeta}{r}\partial_{\zeta}.
\end{eqnarray*}
On the other hand, one can also view the  hyperbolic space $\mathbb{H}^{n+1}$ as a warped product manifold $\RR_+\times \SS^n$ equipped with the metric 
    $$\bar g=d\rho^{2}+\phi^{2}(\rho)\sigma,$$
where $\phi(\rho):=\sinh \rho$.     We denote $r(x)$ and $\rho(x)$ as the geodesic distance from $x$ to the origin in Euclidean space and hyperbolic space respectively.  Then it follows (see e.g. \cite[Section~4.1]{WX2019})
     \begin{eqnarray*}
         V_{0}=\cosh \rho=\frac{1+r^{2}}{1-r^{2}},\quad \sinh \rho=\frac{2r}{1-r^{2}},
     \end{eqnarray*}
     and
\begin{eqnarray}\label{relation}
    r=\frac{e^{\rho}-1}{e^{\rho}+1}.
\end{eqnarray}
 The position vector $x$ in $\HH^{n+1}$ (in polar coordinate) can be represented as
 \begin{eqnarray*}
     x=\sinh \rho \partial_{\rho}.
 \end{eqnarray*}
Using \eqref{relation}, the constant vector field $E_{n+1}$ (in polar coordinate) is
\begin{eqnarray}\label{En1}
E_{n+1}=\cos\zeta\partial_{r}-\frac{\sin\zeta}{r}\partial_{\zeta}
    =\frac{\cos\zeta(e^{\rho}+1)^{2}}{2e^{\rho}}\partial_{\rho}-\frac{\sin\zeta(e^{\rho}+1)}{e^{\rho}-1}\partial_{\zeta}.
\end{eqnarray}
Assuming that capillary hypersurface $\Sigma$ is star-shaped to the origin in $\HH^{n+1}$, then we can reparametrize $\Sigma$ as a graph over $\bar{\SS}^{n}_{+}$. Namely, there exists a positive function $\rho$ defined on $\bar{\SS}^{n}_{+}$, such that
\begin{eqnarray*}
    \Sigma=\{\rho(\eta)\eta| \eta\in \bar{\SS}^{n}_{+}\}.
\end{eqnarray*}
Define a new function $\varphi:\bar\SS^n_+\to\RR$ by $$\varphi(\eta):=\Phi(\rho(\eta)),$$ where  
    \begin{eqnarray*}
    	\frac{d \Phi (\rho) }{d \rho}=\frac{1}{\phi(\rho)}.   \end{eqnarray*}
 Let $\eta:=(\eta^1,\cdots, \eta^n)$ be a local coordinate system of $\bar\SS^n_+$, we write $\p_i:=\p_{\eta^i}$,  
 $\varphi_{i}=\ov{\nabla}_{\partial_{i}}\varphi$ and $\varphi_{ij}=\ov{\nabla}_{\partial_{i}}\ov{\nabla}_{\partial_{j}}\varphi$, where $\ov{\nabla} $ is the Levi-Civita connection on $\bar{\mathbb{S}}^{n}_{+}$. We denote $\sigma_{ij}=\sigma(\partial_{i},\partial_{j}), \varphi^{i}=\sigma^{ij}\varphi_{j}$, then the unit outward normal vector of $\Sigma$ is given by %can be represented as 
    \begin{eqnarray}\label{normal vector}
    	\nu=\frac{1}{\omega} \left(\partial_{\rho} -\frac{\varphi^{j}}{\phi}\partial_{j}\right),
    \end{eqnarray}
    where $\omega:=\sqrt{1+|\ov{\nabla}\varphi|^{2}}$. 

    Let $X_{i}$ denote the vector $\partial_{i}+\rho_{i}\partial_{\rho}$, then $\{X_{i}\}_{i=1}^{n}$ forms a basis of the tangent space of $\Sigma$. The induced metric of $\Sigma$ can be represented as 
    \begin{eqnarray*}
    	g_{ij}=\bar g(X_{i},X_{j})=\phi^{2}\sigma_{ij}+\rho_{i}\rho_{j}=\phi^{2}(\sigma_{ij}+\varphi_{i}\varphi_{j}),
    \end{eqnarray*}
   and its inverse is given by $$g^{ij}=
    \frac{1}{\phi^{2}} \left(\sigma^{ij}-\frac{\varphi^{i}\varphi^{j}}{\omega^{2}} \right).$$
    The second fundamental form of $\Sigma$ is given by
    \begin{eqnarray*}
    	h_{ij}=-\bar g(D_{X_{i}}X_{j},\nu)=-\frac{\phi}{\omega}\left(\varphi_{ij}-\phi^{'}(\sigma_{ij}+\varphi_{i}\varphi_{j})\right),
    \end{eqnarray*}
   and
    \begin{eqnarray*}
    	h_{j}^{i}=g^{ik}h_{kj}=-\frac{1}{\omega \phi}\left(\widehat{\sigma}^{ik}\varphi_{kj}-\phi^{'}(\rho)\delta_{ij} \right),
    \end{eqnarray*}
    where $\widehat{\sigma}^{ik}:=\sigma^{ik}-\frac{\varphi^{i}\varphi^{k}}{\omega^{2}}$.
 The support function of $\Sigma$ is 
    \begin{eqnarray*}
    \bar{g}(x, \nu)=\frac{\phi}{\omega}.
    \end{eqnarray*}
     Combining \eqref{relation}, \eqref{En1} and \eqref{normal vector}, we have
     \begin{eqnarray*}
         \bar{g}(E_{n+1},\nu)=\frac{\cos\zeta (e^{\rho}+1)^{2}}{2e^{\rho}\omega}+\frac{\sin\zeta(e^{\rho}+1)\ov{\n}_{\partial_\zeta}\rho}{\omega(e^{\rho}-1)},
     \end{eqnarray*}
     and 
     \begin{eqnarray*}
         \delta(x, E_{n+1})=r\cos\zeta =\frac{\cos\zeta(e^{\rho}-1)}{e^{\rho}+1}.
     \end{eqnarray*}
  Therefore we obtain 
     \begin{eqnarray*}\label{f repre}
             f&=&\frac{V_{0}-\cos\theta\bar{g}(Y_{n+1},\nu)}{F}-\bar{g}(x,\nu)\notag\\
             &=&\frac{H_{n-1}}{H_{n}}\Bigg[\phi^{'}-\frac{\cos\theta}{2}\left(1+\left(\frac{e^{\rho}-1}{e^{\rho}+1} \right)^{2}\right)\left(\frac{\cos\zeta (e^{\rho}+1)^{2}}{2e^{\rho}\omega}+\frac{\sin\zeta(e^{\rho}+1)\phi\ov{\n}_{\partial_\zeta}\varphi}{\omega(e^{\rho}-1)}\right)\notag\\
             &&+\frac{\cos\theta\cos\zeta(e^{\rho}-1)}{e^{\rho}+1}\frac{\phi}{\omega}\Bigg]-\frac{\phi}{\omega}.
            \end{eqnarray*}
     Along $\partial \SS^{n}_{+}$, i.e. $\zeta=\frac\pi 2$, we have
     \begin{eqnarray*}
         \ov{N}\circ x=\frac{1}{\phi}\partial_{\zeta},
     \end{eqnarray*}
it follows that 
     \begin{eqnarray*}
         -\cos\theta=\bar{g}(\nu, \ov{N}\circ x)=\bar{g}\left(\frac{1}{\omega}\left(\partial_{\rho} -\frac{\varphi^{j}}{\phi} \partial_{j} \right), \frac{1}{\phi}\partial_{\zeta}
         \right)=-\frac{\ov{\nabla}_{\partial_\zeta}\varphi}{\omega},
     \end{eqnarray*}
    is equivalent to 
      \begin{eqnarray*}\label{oblique boundary}
          \ov{\n}_{\partial_{\zeta}}\varphi= \cos\theta
     \sqrt{1+|\ov{\n}\varphi|^{2}}. 
      \end{eqnarray*}
In summary, we can transform the flow \eqref{flow with capillary}
   into the following scalar parabolic flow on $\bar{\SS}^{n}_{+}$ with an oblique boundary value condition
   \begin{equation}\label{scalr flow with capillary}
		\left\{ \begin{array}{lll}
			&\partial_{t}\varphi=\frac{\omega}{\phi}f:=G(\ov\n^{2}\varphi, \ov\n\varphi, \rho, \zeta) ,  &
			\hbox{ in }\SS^{n}_{+}\times[0,T),\\
			&\ov{\n}_{\partial_{\zeta}}\varphi = \cos\theta(1+|\ov{\n}\varphi|^{2})^{\frac{1}{2}} ,
			 & \hbox{ on }\partial \SS^{n}_{+} \times [0,T),\\			
    & \varphi(\cdot,0)  = \varphi_{0}(\cdot),  & \text{ on }  \SS_{+}^{n}.
		\end{array}\right.
	\end{equation}	
  Note that  
 \begin{eqnarray}\label{Y norm}
 \bar{g}(Y_{n+1}, Y_{n+1})
 &=&\frac{(1+|x|^{2})^{2}-4[\delta(x, E_{n+1})]^{2}}{(1-|x|^{2})^{2}},
     \end{eqnarray}
and $\delta( x, E_{n+1}) > 0$ on $\text{int}(\Sigma_{t})$, it is easy to see that  
\begin{eqnarray}
  \bar{g}(Y_{n+1},\nu)\leq \sqrt{\bar g(Y_{n+1}, Y_{n+1})}< V_{0},\label{ellip}  
\end{eqnarray}then $$V_0- \cos\theta \bar g(Y_{n+1},\nu)>0.$$ Therefore, the scalar flow  \eqref{scalr flow with capillary} is strictly parabolic and hence the short-time existence for flow \eqref{flow with capillary} follows from the standard parabolic theory.

 \subsection{Evolution equations}
In order to derive the evolution equations for various geometric quantities, for convenience, we introduce the linearized operator with respect to flow \eqref{flow with capillary} as
\begin{eqnarray}\label{L operator}
    \mathcal{L}:=\partial_{t}-\frac{V_{0}-\cos\theta\bar{g}(Y_{n+1},\nu)}{F^{2}}F^{ij}\n_{i}\n_{j}-\bar{g}\left(x+\T+\frac{\cos\theta }{F} Y_{n+1}, \n\right),
\end{eqnarray}
and denote $\mathcal{F}:=\sum\limits_{i=1}^{n}\frac{\partial F}{\partial h_{i}^{i}}$.  
From Lemma \ref{prop2.1}, we know that $F=\frac{H_n}{H_{n-1}}$ satisfies 
\begin{eqnarray}\label{formula 1}
  \mathcal{F}-\frac{F^{ij}h_{ij}}{F}=\mathcal{F}-1\geq 0, ~~~~~F^{ij}h_{i}^{k}h_{kj}= F^{2}.
\end{eqnarray}
For the conformal Killing vector field $Y_{n+1}$ in \eqref{Y}, the following identities are useful for us, see e.g.  \cite[Proposition~4.3 and Proposition~4.6]{WX2019}.
\begin{prop}[\cite{WX2019}]
   Let $\{e_{i}\}_{i=1}^{n}$ be an orthonormal frame on $\Sigma$, then
    \begin{eqnarray}
    &&\n_{i}Y_{n+1}= e^{-u}   \bar{g}(x, e_{i}) E_{n+1}-e^{-u} \bar{g}(e_{i},  E_{n+1})x  ,\label{Y one}\\
        &&\n_{i}\n_{j}\left(\bar{g}(Y_{n+1}, \nu)\right)= \bar{g}(Y_{n+1}, \n h_{ij})- \bar{g}(Y_{n+1},\nu) (h^{2})_{ij}+\bar{g}(Y_{n+1},\nu)g_{ij}.~~~~~~\label{Y second}
    \end{eqnarray}
\end{prop} 

Now we derive the evolution equations for the induced metric and second fundamental form along the flow \eqref{flow with capillary}.
\begin{prop}
    Along the flow \eqref{flow with capillary}, there holds  
    \begin{eqnarray}\label{evo metric}
        \partial_{t}g_{ij}=2\left(\frac{V_{0}-\cos\theta\bar{g}(Y_{n+1},\nu)}{F}-\bar{g}(x,\nu)\right)h_{ij}+\n_{i}\T_{j}+\n_{j}\T_{i},
    \end{eqnarray}
    and 
    \begin{eqnarray}\label{evo hij}
    \begin{aligned}
      \partial_{t}h_{ij}=& \frac{\left(V_{0}-\cos\theta\bar{g}(Y_{n+1},\nu)\right)}{F^{2}}\left(F^{kl}\n_{k}\n_{l}h_{ij}+ F^{kl,pq}\n_{i}h_{kl}\n_{j}h_{pq}\right)\\
        &+\frac{1}{F^{2}}\left(\n_{i}(V_{0}-\cos\theta\bar{g}(Y_{n+1},\nu))\n_{j}F
    +\n_{j}(V_{0}-\cos\theta\bar{g}(Y_{n+1},\nu))\n_{i}F\right)\\
     &-2\left(V_{0}-\cos\theta\bar{g}(Y_{n+1},\nu)\right)\frac{\n_{i}F\n_{j}F}{F^{3}}+\bar{g}\left(x+\T+\frac{\cos\theta }{F}Y_{n+1},\n h_{ij}\right)\\
     &+\left(\frac{\left(V_{0}-\cos\theta\bar{g}(Y_{n+1},\nu)\right)}{F^{2}}F^{kl}\left((h^{2})_{kl}+g_{kl}\right)+V_{0}+\frac{\bar{g}(x,\nu)}{F}\right)h_{ij}\\
     &-\left(2\bar{g}(x,\nu)+\frac{\cos\theta\bar{g}(Y_{n+1},\nu)}{F}\right)(h^{2})_{ij}-\left(\frac{V_{0}-\cos\theta\bar{g}(Y_{n+1},\nu)}{F}+\bar{g}(x,\nu)\right){g}_{ij}\\
     &+h_{j}^{k}\n_{i}\T_{k}+h_{i}^{k}\n_{j}\T_{k}.
     \end{aligned}
    \end{eqnarray}
\end{prop}
\begin{proof}
     \eqref{evo metric}  is obvious from Proposition \ref{basic evolution eq} (1). In order to show \eqref{evo hij}, by direct calculations,  
\begin{eqnarray*}
    -\n_{i}\n_{j}f&=&-\n_{i}\n_{j}\left(\frac{V_{0}-\cos\theta\bar{g}(Y_{n+1},\nu)}{F}-\bar{g}(x,\nu)\right)\\
    &=&
\left(V_{0}-\cos\theta\bar{g}(Y_{n+1},\nu)\right)\frac{\n_{i}\n_{j}F}{F^{2}}-2\left(V_{0}-\cos\theta\bar{g}(Y_{n+1},\nu)\right)\frac{\n_{i}F\n_{j}F}{F^{3}}\\
    &&+\frac{1}{F^{2}}\Big(\n_{i}\left(V_{0}-\cos\theta\bar{g}(Y_{n+1},\nu)\right)\n_{j}F+\n_{j}\left(V_{0}-\cos\theta\bar{g}(Y_{n+1},\nu)\right)\n_{i}F\Big)\\
    &&-\frac{1}{F}\n_{i}\n_{j}\left(V_{0}-\cos\theta\bar{g}(Y_{n+1},\nu)\right)
+ \n_{i}\n_{j}\bar{g}(x,\nu).
\end{eqnarray*}
Using Simons' type identity (see e.g. \cite[Eq. (2-7)]{A94}), %it holds %and the fact that $F$ is homogeneous of degree 1, we obtain
\begin{eqnarray*}\n_{j}\n_{i}F&=&F^{kl}\n_{k}\n_{l}h_{ij}+F^{kl}\left((h^{2})_{kl}+g_{kl}\right)h_{ij}-F\left((h^{2})_{ij}+g_{ij}\right)\\
    &&+F^{kl,pq}\n_{i}h_{kl}\n_{j}h_{pq},
\end{eqnarray*}
it follows  
    \begin{eqnarray}\label{f tow deri}
    \begin{aligned}
        -\n_{i}\n_{j}f
        =&\frac{\left(V_{0}-\cos\theta\bar{g}(Y_{n+1},\nu)\right)}{F^{2}}\Big(F^{kl}\n_{k}\n_{l}h_{ij}+ F^{kl,pq}\n_{i}h_{kl}\n_{j}h_{pq}\Big)\\
       & +\frac{1}{F^{2}}\Big(\n_{i}\left(V_{0}-\cos\theta\bar{g}(Y_{n+1},\nu)\right)\n_{j}F
    +\n_{j}\left(V_{0}-\cos\theta\bar{g}(Y_{n+1},\nu)\right)\n_{i}F\Big)\\
     &-2\left(V_{0}-\cos\theta\bar{g}(Y_{n+1},\nu)\right)\frac{\n_{i}F\n_{j}F}{F^{3}}+\frac{1}{F^{2}}\left(V_{0}-\cos\theta\bar{g}(Y_{n+1},\nu)\right)\\
     &\cdot\Big(F^{kl}\big((h^{2})_{kl}+g_{kl}\big)h_{ij}-F\big((h^{2})_{ij}+g_{ij}\big)\Big)+\n_{i}\n_{j}\bar{g}(x,\nu)\\
     &-\frac{1}{F}\big(\n_{i}\n_{j}V_{0}-\cos\theta\n_{i}\n_{j}\bar{g}(Y_{n+1},\nu)\big).
    \end{aligned}
    \end{eqnarray}
Recall that (see e.g. Guan-Li \cite[Lemma 2.2 and Lemma 2.6]{GL2014}), 
    \begin{eqnarray}\label{V0 teo deri}
        \n_{i}\n_{j}V_{0}=V_{0}{g}_{ij}- \bar{g}(x, \nu) h_{ij},
    \end{eqnarray}
    and 
    \begin{eqnarray}\label{u two deri}
        \n_{i}\n_{j}\bar{g}(x,\nu)=V_{0}h_{ij}+\bar{g}(x, \n h_{ij})- \bar{g}(x,\nu) (h^{2})_{ij}.
    \end{eqnarray}
    Substituting  \eqref{Y second}, \eqref{V0 teo deri} and  \eqref{u two deri}  into  \eqref{f tow deri}, and combining with Proposition \ref{basic evolution eq} (4), we obtain \eqref{evo hij}.   
\end{proof}

We derive the evolution equation for the support function of $\S_t$ in $\HH^{n+1}$ $$v:=\bar g(x,\nu)=\bar{g}(\phi(\rho)\partial_{\rho}, \nu).$$
\begin{prop}
    Along the flow \eqref{flow with capillary}, there holds  
    \begin{eqnarray}
         \mathcal{L}v &= &\left(\frac{F^{ij}(h^2)_{ij}}{F^{2}}-1\right) V_0 v - \frac{ \bar g(x,\n V_0)}{F} -\cos\theta  v \bar g(Y_{n+1},\nu)  \frac{F^{ij}(h^2)_{ij}}{F^{2}}  \notag\\&& + \frac{\cos\theta}{F}    \bar g(\n_{i} Y_{n+1},\nu)  \bar g(x,e_i), \label{evol of v}%=&-F^{-1}\bar{g}(\n V_{0}, \n V_{0})-\cos\theta   \bar{g}(Y_{n+1},\nu) v+\cos\theta F^{-1}\bar{g}(\n_{e_{i}}Y_{n+1}, \nu)\bar{g}(x, e_{i}),~~~~~~~~
     \end{eqnarray}
    and 
    \begin{eqnarray}\label{v bry condi}
        \n_{\mu}v=\cot\theta h(\mu, \mu) v,\quad \text{on}\quad \partial\Sigma_{t}.
    \end{eqnarray}
\end{prop}
\begin{proof}
  Recall that $\phi\partial_{\rho}$ is a conformal Killing vector field in $\HH^{n+1}$ (see e.g. \cite[Eq. (4.1)]{GLW}), then for any  vector field $X$ on $\HH^{n+1}$,% there holds %that 
\begin{eqnarray}\label{conf}
    D_X(\phi\partial_{\rho})=V_{0}X,
\end{eqnarray}
together with Proposition \ref{basic evolution eq} (3),  \eqref{Y one} and \eqref{conf}, we derive 
\begin{eqnarray}
     \partial_{t}v&=&\bar{g}(D_{\p_{t}}(\phi \partial_{\rho}),\nu)+\bar{g}(x,D_{\partial_{t}}\nu) \notag \\
     &=&V_{0} f- \bar{g}(x, \n f)+h(e_{i}, \T)\bar{g}(x, e_{i}) \notag\\
     %&=&v\left(F^{-1}(V_{0}-\cos\theta \bar{g}(Y_{n+1},\nu))-v\right)-F^{-1}|\n V_{0}|^{2}+\cos\theta F^{-1}\bar{g}(Y_{n+1}, \n v)\\
     %&&+\cos\theta \left(\bar{g}(x^{\perp}, x^{\perp})\bar{g}(\nu, e^{-u}E_{n+1})-\cos\theta\bar{g}(x^{\perp}, e^{-u}E_{n+1})\bar{g}(x,\nu))    \right)+\bar{g}(x,\n v)+\bar{g}(T,\n v)\\
     &=&\frac{V_{0}^{2}}{F}-\frac{\cos\theta}{F} V_{0}\bar{g}(Y_{n+1},\nu)-vV_{0}-\frac{1}{F}\bar{g}(x, \n V_{0})+\bar{g}(x,\n v) \notag \\&&+\bar{g}(\T,\n v)
     +\frac{1}{F^{2}}(V_{0}-\cos\theta\bar{g}(Y_{n+1},\nu))\bar{g}(x,\n F)
     +\frac{\cos\theta}{F} \bar{g}(Y_{n+1}, \n v) \notag \\
     &&+\frac{\cos\theta}{F}  \bar{g}(\n_{i}Y_{n+1}, \nu)\bar{g}(x,e_{i}), \label{ partial_t v}
     %\left(\bar{g}(x^{T}, x^{T})\bar{g}(\nu, e^{-u}E_{n+1})-\bar{g}(x^{T}, e^{-u}E_{n+1})\bar{g}(x,\nu)    \right),
     \end{eqnarray}
     %where $x^T$ denotes the tangential projection of $x$ on $T\S$.
     From \eqref{u two deri} and Lemma \ref{prop2.1} (3), we have
     \begin{eqnarray*}
         F^{ij}\n_{i}\n_{j}v&=&F^{ij}(V_{0}h_{ij}+ \bar{g}(x, \n h_{ij})-(h^{2})_{ij}v)\\
         &=&V_{0}F+\bar{g}(x,\n F)-v F^{ij}(h^2)_{ij}.
     \end{eqnarray*}
Substituting the above equation to \eqref{ partial_t v}, by simple rearrangement of some terms, we obtain \eqref{evol of v}. %    \begin{eqnarray*}         \mathcal{L}v&=& \left(\frac{F^{ij}(h^2)_{ij}}{F^2}-1\right) V_0 v -F^{-1} \bar g(x,\n V_0) -\cos\theta  v \bar g(Y_{n+1},\nu)  \frac{ F^{ij}(h^2)_{ij} }{F^2} \\&& +\frac{\cos\theta}{F}   \bar g(\n_{e_i} Y_{n+1},\nu)  \bar g(x,e_i).%         \\ &=&-F^{-1}\bar{g}(\n V_{0}, \n V_{0})-\cos\theta v \bar{g}(Y_{n+1},\nu)+\cos\theta F^{-1}\bar{g}(\n_{e_{i}}Y_{n+1}, \nu)\bar{g}(x, e_{i}),     \end{eqnarray*}%where the last equality used \eqref{formula 1}. 
     On $\partial\Sigma_{t}$,  there holds $\bar g(x, \ov N)=0$, and \eqref{co-normal bundle} implies
     \begin{eqnarray}\label{mu to nu N}
    \mu=\cot\theta\nu+\frac{1}{\sin\theta}\ov{N},
\end{eqnarray} together with Proposition \ref{basic-capillary} (1), we derive
     \begin{eqnarray*}
         \n_{\mu}v=h(\mu,\mu)\bar{g}(x,\mu)=\cot\theta h(\mu,\mu)v.
     \end{eqnarray*}
   Hence the assertions follow.

\end{proof}

\begin{prop}\label{pro F evo}
    Along the flow \eqref{flow with capillary}, there holds
    \begin{eqnarray}
    \mathcal{L}F&=&-\frac{ 2\left(V_{0}-\cos\theta\bar{g}(Y_{n+1},\nu)\right)}{F^{3}} F^{ij}\n_{i}F\n_{j}F+\frac 2{F^2} F^{ij}\n_{i}F \notag\\
    &&\cdot\n_{j}\left(V_{0}-\cos\theta\bar{g}(Y_{n+1},\nu)\right)+(1-\mathcal{F})v+ \left(1-\frac{F^{ij}(h^{2})_{ij}}{F^2} \right) V_{0} F,\label{evol F}
\end{eqnarray}
and 
\begin{eqnarray}\label{F boundary}
    \n_{\mu}F=0,\quad \text{on}\quad \partial\Sigma_{t}.
\end{eqnarray}
\end{prop}
\begin{proof}
     From Proposition \ref{basic evolution eq} and \eqref{V0 teo deri}, \eqref{u two deri}, we obtain
    \begin{eqnarray*}
        \partial_{t}F&=&-F^{ij}\n_{i}\n_{j}f-f F^{ij}(h^{2})_{ij}+f\sum\limits_{i=1}^{n}F^{ii}+\bar{g}(\n F, \T)\\
        &=&-\frac{1}{F}F^{ij}\n_{i}\n_{j}\left(V_{0}-\cos\theta \bar{g}(Y_{n+1}, \nu)\right)+\frac{2}{F^{2}}F^{ij}\n_{i}(V_{0}-\cos\theta\<Y_{n+1},\nu\>)\n_{j}F\\
        &&+\frac{1}{F^{2}}\left(V_{0}-\cos\theta\bar{g}(Y_{n+1},\nu)\right)F^{ij}\n_{i}\n_{j}F-\frac{2}{F^{3}}\left(V_{0}-\cos\theta\bar{g}(Y_{n+1},\nu)\right)F^{ij}\n_{i}F\n_{j}F\\
        &&+F^{ij}\n_{i}\n_{j}v
        -f F^{ij}(h^{2})_{ij}+f\sum\limits_{i=1}^{n}F^{ii}+\bar{g}(\n F, \T)\\
        %&=&-\frac{V_{0}\sum\limits_{i=1}^{n}F^{ii}-F\bar{g}(x,\nu)-\cos\theta\left(\bar{g}(Y_{n+1}, \n F)-F^{ij}(h^{2})_{ij}\bar{g}(Y_{n+1},\nu)+\sum\limits_{i=1}^{n}F^{ii}\bar{g}(Y_{n+1},\nu)\right)}{F}\\
        %&&+\frac{2F^{ij}\left(V_{0}-\cos\theta\bar{g}(Y_{n+1},\nu)\right)_{i}F_{;j}}{F^{2}}+\frac{(V_{0}-\cos\theta\bar{g}(Y_{n+1},\nu))}{F^{2}}F^{ij}F_{;ij}-\frac{2(V_{0}-\cos\theta\bar{g}(Y_{n+1},\nu))F^{ij}F_{;i}F_{;j}}{F^{3}}\\
        %&&-fF^{ij}(h^{2})_{ij}+f\sum\limits_{i=1}^{n}F^{ii}
        %+\bar{g}(\n F, T)+V_{0}F+\bar{g}(x, \n F)-\bar{g}(x,\nu)F^{2}\\
        &=&(V_{0}-\cos\theta\bar{g}(Y_{n+1},\nu)\frac{F^{ij}\n_{i}\n_{j}F}{F^{2}}
        +\bar{g} \left(x+\T+\frac{\cos\theta Y_{n+1}}{F}, \n F \right)
        +\frac{2}{F^{2}}F^{ij}\\        
        &&\cdot\n_{i}(V_{0}-\cos\theta\bar{g}(Y_{n+1},\nu))\n_{j}F
        -\frac{2}{F^{3}}\left(V_{0}-\cos\theta\bar{g}(Y_{n+1},\nu)\right)F^{ij}\n_{i}F\n_{j}F\\
        &&
        +(1-\mathcal{F})v+\frac{V_{0}}{F}(F^{2}-F^{ij}(h^{2})_{ij}),
    \end{eqnarray*}taking into account of \eqref{L operator}, then \eqref{evol F} follows. 
    Along $\partial\S_{t}$, from \eqref{mu to nu N}, \begin{eqnarray}\label{V0 bry}
    \n_{\mu}V_{0}=\bar{g}(x,\mu)=\cot\theta\bar{g}(x,\nu).
\end{eqnarray}
 Note that $\ov{N}=-e^{-u}E_{n+1}$, %direct calculations yield 
\begin{eqnarray}
%\begin{aligned}
    \bar{g}(Y_{n+1},\ov{N})&=&\frac{-4}{(1-|x|^{2})^{2}}\left(\frac{1}{2}(1+|x|^{2})\delta(E_{n+1}, e^{-u}E_{n+1})-\delta(x, E_{n+1})\delta(x, e^{-u}E_{n+1})\right)\notag \\
   & =&-\frac{1+|x|^{2}}{1-|x|^{2}}=-V_{0}. ~~~~\label{bry deri of Y}
 %   \end{aligned}
\end{eqnarray}
Combining with Proposition \ref{basic-capillary}, \eqref{Y one} and \eqref{bry deri of Y}  , we obtain
\begin{eqnarray*}
    \n_{\mu}\bar{g}(Y_{n+1},\nu)&=&h(\mu, \mu)\bar{g}(Y_{n+1}, \mu)+e^{-u}\big[\bar{g}(x,\mu)\bar{g}(\nu, E_{n+1})-\bar{g}(\mu, E_{n+1})\bar{g}(x, \nu)\big]\\
    &=&\cot\theta h(\mu,\mu)\bar{g}(Y_{n+1},\nu)+\frac{1}{\sin\theta}h(\mu,\mu)\bar{g}(Y_{n+1},\ov{N})-\bar{g}(x,\mu)\bar{g}(\nu, \ov{N})\\
    &&+\bar{g}(\mu,\ov{N})\bar{g}(x,\nu)\\
   % &=&\cot\theta h(\mu,\mu)\bar{g}(Y_{n+1},\nu)+\frac{1}{\sin\theta}h(\mu,\mu)\bar{g}(Y_{n+1},\bar{N})+\frac{1}{\sin\theta}\bar{g}(x,\nu)\\
    &=&\cot\theta h(\mu,\mu)\bar{g}(Y_{n+1},\nu)-\frac{V_{0}}{\sin\theta}h(\mu,\mu)+\frac{1}{\sin\theta}\bar{g}(x,\nu),
\end{eqnarray*}
together with \eqref{V0 bry}, it yields
\begin{eqnarray}\label{bry condi 2}
    \n_{\mu}\left(V_{0}-\cos\theta\bar{g}(Y_{n+1},\nu)\right)
    %&=&\cot\theta\bar{g}(x,\nu)-\cos\theta\left(\cot\theta h(\mu,\mu)\bar{g}(Y_{n+1},\nu)-\frac{V_{0}}{\sin\theta}h(\mu,\mu)+\frac{1}{\sin\theta}\bar{g}(x,\nu)\right)\\
    =\cot\theta h(\mu,\mu)\left(V_{0}-\cos\theta \bar{g}(Y_{n+1},\nu)\right).
\end{eqnarray}
From \eqref{boundary condition}, \eqref{v bry condi} and \eqref{bry condi 2}, we obtain the assertion \eqref{F boundary}, since %on $\partial\S_{t}$, 
\begin{eqnarray*}
    \n_{\mu}F=\n_{\mu}\left(\frac{V_{0}-\cos\theta\bar{g}(Y_{n+1},\nu)} {f+\bar{g}(x,\nu)}\right)=0.
\end{eqnarray*}
\end{proof}

Next, we  introduce the \textit{capillary support function} $\wt v$ for the capillary hypersurface $\S\subset \HH^{n+1}$. By decomposing the position vector $x$ with respect to the capillary outward normal $\wt \nu$ in \eqref{capillary normal nu} as 
\begin{eqnarray*}
    x= (V_0{\wt v} )\wt \nu+W,
\end{eqnarray*}for some tangential vector field $W$ of $\S$, and $\wt v$ is given by
%that will play an important role in deriving the lower bound of the curvature function $F$. Its definition is as follows
\begin{eqnarray}\label{wt v}
    \widetilde{v}:=\frac 1 {V_0} \frac{\bar g(x,\nu)}{ \bar{g}(\wt \nu,\nu)}=\frac{v}{V_{0}-\cos\theta\bar{g}(Y_{n+1},\nu)}.\end{eqnarray} %where the last equality used \eqref{capillary normal nu}. 
This function will play an important role in deriving the curvature estimates along flow \eqref{flow with cap-normal} or \eqref{flow with capillary} later. %We call $\widetilde{v}$ as the capillary support function of $\S\subset \ol{\H^+}$. Previously, we introduced a similar function for studying capillary hypersurfaces in the Euclidean half-space, see for instance the capillary support function as in \cite[Eq. 3.3]{MWW}. 
This function is a very natural analogous notion of the classical support function for capillary hypersurface $\S\subset \HH^{n+1}$ in the sense that $\wt v$ is identically constant, given by $\frac{2r_{0}}{1 + r_{0}^{2}\sin^{2}\theta}$, when $\S$ is the 
 geodesic spherical cap $\C_{\theta, r_{0}}$ in $\HH^{n+1}$, by using Proposition \ref{sup of model}. One can also refer to a similar concept called the relative support function, which was introduced by \cite[Section 3. Remark]{AM} in an anisotropic setting. %Besides, a similar function was also introduced and crucially used in the prescribed curvature problems by \cite[Eq. (1.6)]{MWW23}.  
 \begin{prop}\label{pro wide v}
 Along the flow \eqref{flow with capillary}, % the capillary support function 
 $\widetilde{v}$ in \eqref{wt v} satisfies
 \begin{eqnarray}
     \mathcal{L}\widetilde{v}&=&\frac{2}{F^{2}}F^{ij}\n_{i}\widetilde{v}\n_{j}(V_{0}-\cos\theta \bar{g}(Y_{n+1}, \nu))-\frac{2}{F}\widetilde{v}^{2}(V_{0}-\cos\theta\bar{g}(Y_{n+1}, \nu))+\frac{v}{F^{2}}\mathcal{F}\notag \\
   &&+\widetilde{v}\Bigg[\left(\frac{F^{ij}(h^{2})_{ij}}{F^{2}}-1\right)V_{0}
   %&&
   %-\frac{\cos\theta }{F^{2}}\bar{g}(Y_{n+1}, \nu)F^{ij}(h^{2})_{ij}+\cos\theta %\bar{g}(Y_{n+1}, \nu)F^{ij}(h^{2})_{ij}\notag \\
   +\frac{1}{V_{0}-\cos\theta\bar{g}(Y_{n+1}, \nu)} 
   \big(\bar{g}(x,x)-\cos\theta \bar{g}(\n_{i}Y_{n+1}, \nu)\bar{g}(x,e_{i})\big)\Bigg]\notag \\
   &&-\frac{1}{F(V_{0}-\cos\theta\bar{g}(Y_{n+1}, \nu))}\left[\bar{g}(x,\n V_{0})-\cos\theta\bar{g}(\n_{i}Y_{n+1}, \nu)\bar{g}(x, e_{i})\right],  \label{wide v evo}
    \end{eqnarray}
and 
\begin{eqnarray}\label{widetilde v boundary}
    \n_{\mu}\widetilde{v}=0,\quad \text{on}\quad \partial\Sigma_{t}. 
\end{eqnarray}
\end{prop}
\begin{proof}
By direct calculations, 
    \begin{eqnarray*}
        \partial_{t}V_{0}=\bar{g}(D V_{0}, f\nu+\T)=f v+\bar{g}(\nabla V_{0}, \T),
    \end{eqnarray*}
from \eqref{V0 teo deri}, we have
    \begin{eqnarray*}
        F^{ij}\n_{i}\n_{j}V_{0}=V_{0}\mathcal{F}-vF,
    \end{eqnarray*}
    then 
    \begin{eqnarray}
        \L V_0&= &-\frac{V_0-\cos\theta \bar g(Y_{n+1},\nu)}{F^2} V_0\mathcal F + \frac{2 (V_0-\cos\theta \bar g(Y_{n+1},\nu))}{F}v-v^2\notag \\&& -\bar g \left(x+\frac{\cos\theta}{F} Y_{n+1}, \n V_0 \right). \label{V0 evo equ}
    \end{eqnarray} %Hence    \begin{eqnarray}    \begin{aligned}\label{V0 evo equ}    &\partial_{t}V_{0}-F^{-2}\left(V_{0}-\cos\theta\bar{g}(Y_{n+1},\nu)\right)F^{ij}\n_{i}\n_{j}V_{0}\\    =&\bar{g}(\T,\n V_{0})-F^{-2}(V_{0}-\cos\theta\bar{g}(Y_{n+1},\nu))V_{0}\mathcal{F}+2F^{-1}V_{0}v    -2\cos\theta v F^{-1}\bar{g}(Y_{n+1},\nu)-v^{2}.~~~    \end{aligned}    \end{eqnarray}
Recall that $Y_{n+1}$ is a Killing vector field  (see e.g. \cite[Proposition~4.1]{WX2019}), it follows
\begin{eqnarray*}
    \bar g(D_\nu Y_{n+1},\nu)=0,
\end{eqnarray*}
together with Proposition \ref{basic evolution eq} (3) and 
\eqref{Y one}, we derive
    \begin{eqnarray*}
      \partial_{t}\bar{g}(Y_{n+1},\nu)&=&\bar{g}(D_{f\nu+\T} (Y_{n+1}),\nu)+\bar{g}(Y_{n+1},D_{\partial_{t}} \nu)\\
      &=&\bar{g}(D_{\T}Y_{n+1},\nu)-\bar{g}(Y_{n+1},\n f)+h(e_{i}, \T)\bar{g}(Y_{n+1}, e_{i})\\
      &=&\bar{g}\left(\T,\n (\bar{g}(Y_{n+1},\nu))\right)-\frac{1}{F}\bar{g}(\n (V_{0}-\cos\theta \bar{g}(Y_{n+1},\nu)), Y_{n+1})\\
      %-\cos\theta\bar{g}(\n\bar{g}(Y_{n+1},\nu), Y_{n+1})\right)\\
      &&+\frac{1}{F^{2}}\left(V_{0}-\cos\theta \bar{g}(Y_{n+1},\nu)\right)\bar{g}(Y_{n+1},\n F)+\bar{g}\left(x, \n \bar{g}(Y_{n+1},\nu)\right)\\
      &&-\bar{g}(\n_{i}Y_{n+1},\nu)\bar{g}(x, e_{i}),
    \end{eqnarray*}
and from  \eqref{Y second},
    \begin{eqnarray*}
        F^{ij}\n_{i}\n_{j}(\bar{g}(Y_{n+1},\nu))&=&\bar{g}(Y_{n+1},\n F)-\bar{g}(Y_{n+1},\nu)F^{ij}(h^{2})_{ij}+\bar{g}(Y_{n+1},\nu)\mathcal{F},
    \end{eqnarray*}
then it follows 
    \begin{eqnarray}\label{Yn1 evo equ}
    \mathcal{L}(\bar{g}(Y_{n+1},\nu)) &=&
    \frac{V_{0}-\cos\theta\bar{g}(Y_{n+1},\nu)}{F^{2}}
\bar{g}(Y_{n+1}, \nu)\left(F^{ij}(h^{2})_{ij}-\mathcal{F}\right)\notag \\
&&-\frac{\bar{g}(\n V_{0},Y_{n+1})}{F}
    -\bar{g}(\n_{i}Y_{n+1},\nu)\bar{g}(x,e_{i}).
         \end{eqnarray}
         Combining \eqref{V0 evo equ} and \eqref{Yn1 evo equ}, we get
         \begin{eqnarray*}
            && \mathcal{L}(V_{0}-\cos\theta\bar{g}(Y_{n+1},\nu))=\L V_0-\cos\theta \L \bar g(Y_{n+1},\nu) \\
            &=& -\frac{(V_{0}-\cos\theta \bar{g}(Y_{n+1},\nu))^{2}}{F^{2}}\mathcal{F}-\bar{g}(x, x)+\frac{2\left(V_{0}-\cos\theta\bar{g}(Y_{n+1}, \nu)\right)v}{F}\\
            &&+\cos\theta\bar{g}(\n_{i}Y_{n+1}, \nu)\bar{g}(x,e_{i})-\cos\theta(V_{0}-\cos\theta\bar{g}(Y_{n+1},\nu))\bar{g}(Y_{n+1},\nu)\frac{F^{ij}(h^{2})_{ij}}{F^{2}},
         \end{eqnarray*}
together with  \eqref{evol of v}, it implies
\begin{eqnarray*}
    \mathcal{L}\widetilde{v}&=&\frac{1}{V_{0}-\cos\theta\bar{g}(Y_{n+1}, \nu)}\mathcal{L}v-\frac{v}{\left(V_{0}-\cos\theta\bar{g}(Y_{n+1},\nu)\right)^{2}}\mathcal{L}\left(V_{0}-\cos\theta\bar{g}(Y_{n+1},\nu)\right)\\
    &&+\frac{2}{F^{2}}F^{ij}\n_{i}\widetilde{v}\n_{j}(V_{0}-\cos\theta\bar{g}(Y_{n+1},\nu))\\
    %&=&(\frac{F^{ij}(h^{2})_{ij}}{F^{2}}-1)V_{0}\widetilde{v}-\frac{\cos\theta}{F^{2}} \widetilde{v} \bar{g}(Y_{n+1}, \nu)F^{ij}(h^{2})_{ij}-\frac{1}{F(V_{0}-\cos\theta\bar{g}(Y_{n+1}, \nu))}\\ &&\cdot\Big(\bar{g}(x, \n V_{0})
 %   -\cos\theta \bar{g}(\n_{i}Y_{n+1}, \nu)\bar{g}(x,e_{i})\Big)
  % +\frac{1}{F^{2}}v\mathcal{F}+\frac{\widetilde{v}}{V_{0}-\cos\theta \bar{g}(Y_{n+1}, \nu)}\bar{g}(x,x)\\
  % &&-\frac{2}{F}\widetilde{v}^{2}(V_{0}-\cos\theta\bar{g}(Y_{n+1}, \nu))-\frac{\cos\theta\widetilde{v} }{V_{0}-\cos\theta\bar{g}(Y_{n+1}, \nu)}\bar{g}(\n_{i}Y_{n+1}, \nu)\bar{g}(x, e_{i})\\
  % &&+\cos\theta \widetilde{v}\bar{g}(Y_{n+1}, \nu)F^{ij}(h^{2})_{ij}-\frac{2}{F^{2}}F^{ij}\n_{i}\widetilde{v}\n_{j}(V_{0}-\cos\theta\bar{g}(Y_{n+1}, \nu))\\
   &=&\frac{2}{F^{2}}F^{ij}\n_{i}\widetilde{v}\n_{j}(V_{0}-\cos\theta \bar{g}(Y_{n+1}, \nu))-\frac{2}{F}\widetilde{v}^{2}(V_{0}-\cos\theta\bar{g}(Y_{n+1}, \nu))+\frac{1}{F^{2}}v\mathcal{F}\\
   &&+\widetilde{v}\left[ \left(\frac{F^{ij}(h^{2})_{ij}}{F^{2}}-1\right)V_{0}+\frac{1}{(V_{0}-\cos\theta\bar{g}(Y_{n+1}, \nu))} 
   \big(\bar{g}(x,x)-\cos\theta \bar{g}(\n_{i}Y_{n+1}, \nu)\bar{g}(x,e_{i})\big)\right]\\ 
   &&-\frac{1}{F(V_{0}-\cos\theta\bar{g}(Y_{n+1}, \nu))}\big[\bar{g}(x,\n V_{0})-\cos\theta\bar{g}(\n_{i}Y_{n+1}, \nu)\bar{g}(x, e_{i})\big]. 
    \end{eqnarray*}
       % \begin{eqnarray*}
   %&=&(F^{-1}F^{ij}(h^{2})_{ij}-1)V_{0}\widetilde{v}-\left(V_{0}-\cos\theta\bar{g}(Y_{n+1},\nu)\right)^{-1}F^{-1}\Big(\bar{g}(\n V_{0}, \n V_{0})+\cos\theta\bar{g}(\n_{i}Y_{n+1},\nu)\bar{g}(x,e_{i})\Big)\\
   % &&+\widetilde{v}\left(V_{0}-\cos\theta\bar{g}(Y_{n+1},\nu)\right)^{-1}\bar{g}(x,x)-2F^{-1}V_{0}\widetilde{v}^{2}+v F^{-2}\mathcal{F}
    %\\%-2\bar{g}^{2}(x,\nu)F^{-1}\left(V_{0}-\cos\theta\bar{g}(Y_{n+1},\nu)\right)^{-2}\\
    %&&+2\cos\theta F^{-1}\bar{g}(x,\nu)\left(V_{0}-\cos\theta\bar{g}(Y_{n+1},\nu)\right)^{-2}-\cos\theta \bar{g}(x,\nu)(V_{0}-\cos\theta\bar{g}(Y_{n+1},\nu))^{-2}\\
    %&&\cdot\bar{g}(\n_{e_{i}}Y_{n+1},\nu)\bar{g}(x, e_{i}).
    %\end{eqnarray*}
    Along $\partial\Sigma_{t}$, \eqref{widetilde v boundary} follows directly from \eqref{v bry condi} and \eqref{bry condi 2}. Hence we proved the assertions. % implies    \begin{eqnarray*}        \n_{\mu}\widetilde{v}=0.    \end{eqnarray*}
\end{proof}

Next, we calculate the evolution equation for the function  $$P:=\widetilde{v}F.$$
\begin{prop}\label{wvF}
    Along the flow \eqref{flow with capillary}, there hold % the function $\widetilde{v}F$ satisfies 
    \begin{eqnarray}
        \mathcal{L} P  &=&\frac{2}{F^{2}}F^{ij}\n_{i}\left(V_{0}-\cos\theta\bar{g}(Y_{n+1},\nu)\right)\n_{j} P-\frac{2\left(V_{0}-\cos\theta\bar{g}(Y_{n+1},\nu)\right)}{F^3}{ F^{ij}\n_{i}F\n_{j}P}  \notag \\
        &&+ \frac{P}{V_{0}-\cos\theta\bar{g}(Y_{n+1}, \nu)}  \Big[ \bar{g}(x,x)-\cos\theta\bar{g}(\n_{i}Y_{n+1}, \nu)\bar{g}(x,e_{i})
             \Big]+v\mathcal{F} \left(\frac{1}{F}-\widetilde{v} \right)\notag\\&&-v\widetilde{v} 
             -\frac{1}{V_{0}-\cos\theta\bar{g}(Y_{n+1}, \nu)}\big[\bar{g}(x, \n V_{0})
             -\cos\theta\bar{g}(\n_{i}Y_{n+1}, \nu)\bar{g}(x,e_{i})\big], \label{P evo}
             %&&-2\widetilde{v}^{2}(V_{0}-\cos\theta\bar{g}(Y_{n+1}, \nu))~~~~~
        \end{eqnarray}
    and 
    \begin{eqnarray}\label{vF bry condi}
        \n_{\mu} P=0,\quad {\rm{on}}\quad \partial\Sigma_{t}.
    \end{eqnarray}
\end{prop}
\begin{proof}
Direct calculations yield 
\begin{eqnarray*}
\mathcal{L}P=F\mathcal{L} \widetilde{v}+\widetilde{v}\mathcal{L} F-\frac{2(V_{0}-\cos\theta \bar{g}(Y_{n+1}, \nu))}{F^{2}}F^{ij}\n_{i}\widetilde{v}\n_{j}F,
\end{eqnarray*} by substituting \eqref{evol F} and \eqref{wide v evo} into above equation, we derive \eqref{P evo}. And \eqref{vF bry condi} follows easily from \eqref{F boundary} and \eqref{widetilde v boundary}. 
\end{proof}

In order to obtain the uniform bound of principal curvatures, we still need to derive the evolution equation for the mean curvature $H$.% along flow \eqref{flow with capillary}.
 \begin{prop}\label{pro H evo}
         Along the flow \eqref{flow with capillary}, $H$ satisfies %the mean curvature $H$ evolves as follows 
         \begin{eqnarray}
             \mathcal{L}H&=&\frac{V_{0}-\cos\theta\bar{g}(Y_{n+1},\nu)}{F^2}{F^{kl,pq}\n_{i}h_{kl}\n_{i}h_{pq}}-\frac{2\left(V_{0}-\cos\theta\bar{g}(Y_{n+1},\nu)\right)}{F^{3}}|\n F|^{2}\notag\\
            &&+\frac{2}{F^{2}}\n_{i}\left(V_{0}-\cos\theta\bar{g}(Y_{n+1},\nu)\right)\n_{i}F-\frac{ \left(2V_{0}-\cos\theta\bar{g}(Y_{n+1},\nu)\right)}{F}|h|^{2}\notag \\
             &&+H\left[\frac{V_{0}-\cos\theta\bar{g}(Y_{n+1}, \nu)}{F^2} {F^{kl}(h^{2})_{kl}}+\frac{V_{0}-\cos\theta\bar{g} (Y_{n+1}, \nu )}{F^{2}}\mathcal{F}+\frac{v}{F}+V_0\right]\notag\\
                 &&-\frac{n\left(V_{0}-\cos\theta\bar{g}(Y_{n+1},\nu)\right)}{F} -nv. \label{H-equ}
         \end{eqnarray}
    If $\theta\in(0, \frac \pi 2]$, then
         \begin{eqnarray}\label{bdry of H}
             \n_{\mu}H\leq 0\quad {\rm{on}}\quad \partial\Sigma_{t}.
         \end{eqnarray}
     \end{prop}
     \begin{proof}
         From Proposition \ref{basic evolution eq} (6), together with \eqref{Y second}, \eqref{V0 teo deri} and \eqref{u two deri}, we compute
         \begin{eqnarray*}
             \partial_{t}H&=&-\Delta f -|h|^{2}f+nf+\bar{g}(\T,\n H)\\
             &=&\frac{1}{F}\left(Hv-nV_{0}\right)+\frac{\cos\theta}{F}\left(\bar{g}(Y_{n+1},\n H)+n\bar{g}(Y_{n+1},\nu)-|h|^{2}\bar{g}(Y_{n+1},\nu) \right)\\
             &&+\frac{2}{F^{2}}\n_{i}\left(V_{0}-\cos\theta\bar{g}(Y_{n+1},\nu)\right)\n_{i}F-\left(V_{0}-\cos\theta\bar{g}(Y_{n+1},\nu)\right)\frac{|\n F|^{2}}{F^{3}}
             \\
             &&+\left(V_{0}-\cos\theta\bar{g}(Y_{n+1},\nu)\right)\frac{\Delta F}{F^{2}}
             +HV_{0}+\bar{g}(x,\n H)
             -\bar{g}(x,\nu)|h|^{2}\\
             &&-(|h|^{2}-n)\left(\frac{V_{0}-\cos\theta\bar{g}(x,\nu)}{F}-v\right)+\bar{g}(\T,\n H)\\
             &=&\left(V_{0}-\cos\theta\bar{g}(Y_{n+1},\nu)\right)\frac{\Delta F}{F^{2}}+\bar{g}\left(x+\T+\frac{\cos\theta Y_{n+1}}{F}, \n H\right)\\
             &&+\frac{2}{F^{2}}\n_{i}\left(V_{0}-\cos\theta\bar{g}(Y_{n+1},\nu)\right)\n_{i}F
            - \frac{2\left(V_{0}-\cos\theta\bar{g}(Y_{n+1},\nu)\right) }{F^{3}} |\n F|^{2}\\
            &&+\frac{v}{F}H+HV_{0}-\frac{1}{F}V_{0}|h|^{2}-nv.
             \end{eqnarray*}
By  Simons' type identity (see e.g. \cite[Eq. (2-7)]{A94}), 
             \begin{eqnarray*}
                 \Delta F&=&F^{kl,pq}\n_{i}h_{kl}\n_{i}h_{pq}+F^{kl}\n_{i}\n_{i}h_{kl}\\
                 &=&F^{kl,pq}\n_{i}h_{kl}\n_{i}h_{pq}+F^{kl}\n_{k}\n_{l}H+H F^{kl}(h^{2})_{kl}-F|h|^{2}-nF+H\sum\limits_{i=1}^{n}F^{ii},
             \end{eqnarray*}
combining the above, we get
             \begin{eqnarray*}
                 \partial_{t}H&=&\left(V_{0}-\cos\theta\bar{g}(Y_{n+1},\nu)\right)\frac{F^{kl}\n_{k}\n_{l}H}{F^{2}}+\bar{g}(x+\T+\frac{\cos\theta Y_{n+1}}{F},\n H)\\
                 &&+\left(V_{0}-\cos\theta\bar{g}(Y_{n+1},\nu)\right)\frac{F^{kl,pq}\n_{i}h_{k}\n_{i}h_{pq}}{F^{2}}+\frac{2}{F^{2}}\n_{i}\left(V_{0}-\cos\theta\bar{g}(Y_{n+1},\nu)\right)\n_{i}F\\
                 &&-2\left(V_{0}-\cos\theta\bar{g}(Y_{n+1},\nu)\right)\frac{|\n F|^{2}}{F^{3}}-\frac{1}{F}\left(2V_{0}-\cos\theta\bar{g}(Y_{n+1},\nu)\right)|h|^{2}\\
                 &&+H\Big(V_{0}+(V_{0}-\cos\theta\bar{g}(Y_{n+1}, \nu))\frac{F^{kl}(h^{2})_{kl}}{F^{2}}+\frac{v}{F}+\frac{1}{F^{2}}(V_{0}-\cos\theta\bar{g}(Y_{n+1},\nu))\mathcal{F}\Big)\\&&
                 -nv -\frac{n}{F}\left(V_{0}-\cos\theta\bar{g}(Y_{n+1},\nu)\right),
             \end{eqnarray*}taking into account of \eqref{L operator}, then \eqref{H-equ} follows.

            Along $\partial\Sigma_{t}$, we choose an orthonormal frame $\{e_{\alpha}\}_{\alpha=2}^{n}$ of $T\partial\Sigma_{t}$ such that $\{e_{1}=\mu, (e_{\alpha})_{\alpha=2}^{n}\}$ forms an orthonormal frame for $T\Sigma_{t}$ and $(h_{ij})$ is diagonal. By  Proposition \ref{basic-capillary}, for any $2\leq \alpha\leq n$, 
           \begin{eqnarray}\label{h-3 de}
    \n_{\mu}h_{\alpha\beta}=\cos\theta\widehat{h}_{\beta\gamma}(h_{11}\delta_{\alpha\gamma}-h_{\alpha\gamma}),
           \end{eqnarray}
         and Proposition \ref{pro F evo} implies 
          \begin{eqnarray}\label{F-1}
0=\n_{\mu}F=F^{11}\n_{{1}}h_{11}+\sum\limits_{\alpha=2}^{n}F^{\alpha\alpha}\n_{{1}}h_{\alpha\alpha}.
          \end{eqnarray}
          together with \eqref{h-3 de} and Proposition \ref{basic-capillary}, we have
          \begin{eqnarray*}
              \n_{\mu}H&=&\n_{1}h_{11}+\sum\limits_{\alpha=2}^{n}\n_{1}h_{\alpha\alpha }=\sum\limits_{\alpha=2}^{n}\left(-\frac{F^{\alpha\alpha}}{F^{11}}\n_{1}h_{\alpha\alpha }+\n_{{1}}h_{\alpha\alpha }\right)  \\
              &=&\sum\limits_{\alpha=2}^{n}\frac{1}{F^{11}}(F^{11}-F^{\alpha\alpha})(h_{11}-h_{\alpha\alpha})\widetilde{h}_{\alpha\alpha}\leq 0,
              \end{eqnarray*}
             where the last inequality follows from Lemma \ref{pro2.3} (3) and the convexity of $\partial\Sigma_{t}\subset \Sigma_{t}$. 
     \end{proof}
 
 \subsection {A priori estimates}\ 
 
 Let $T^{\ast}$ be the maximal time such that there exists a smooth solution to equation \eqref{flow with capillary} on the interval $[0, T^{\ast})$, this implies the strict convexity of $\S_{t}~(0\leq t< T^{\ast})$.
As the origin lies in $\widehat{\partial\Sigma_{0}}$, which indicates  there exists a positive constant $r_{1}$, such that 
$\C_{\theta, r_{1}}\subset \widehat{\Sigma_{0}}$. If there exists a constant $r_2>0$,  such that $\Sigma_{0}\subset \widehat{\C_{\theta, r_2}}$, (by assumption \eqref{add_con}, we can choose $r_2=r_0$ for our flow \eqref{flow with capillary}). From Proposition \ref{sup of model},  the geodesic spherical caps $\C_{\theta, r_{0}}$ are the static solutions to our flow \eqref{flow with capillary}.   With the help of Proposition \ref{sup of model}, following the same argument as in Wang-Weng-Xia \cite[Proposition 4.2 and Proposition 4.10]{WWX1}, we have the following barrier estimate and star-shaped estimate. %which is based on the avoidable principle,  we have 
\begin{prop}\label{C0 estimates}
    For any $t\in [0, T^{\ast})$, the smooth solution $\Sigma_{t}$ of flow \eqref{flow with capillary} satisfies
    \begin{eqnarray}\label{c0 est}
        \Sigma_{t}\subset \widehat{\C_{\theta,r_2}}\setminus\widehat{\C_{\theta,r_1}},
    \end{eqnarray}
    and 
    \begin{eqnarray*}\label{c1 est}
        v=\bar{g}\left(x,\nu\right)\geq C, % \quad \forall ~ (p,t)\in M\times[0,T^{\ast}), 
    \end{eqnarray*}
    where the positive constant $C$ depends only on the initial datum.
\end{prop}
Next, we derive the uniform upper bound of the curvature function $F$.
      \begin{prop}\label{F upper}
          Along the flow \eqref{flow with capillary}, there holds
          \begin{eqnarray*}
              F(p, t)\leq \max\limits_{M}F(\cdot, 0),\quad \forall(p, t)\in M\times[0, T^{\ast}).
          \end{eqnarray*}
      \end{prop}
      \begin{proof}
          The conclusion follows directly from the maximum principle, due to Proposition \ref{pro F evo} and \eqref{formula 1} imply
          \begin{eqnarray*}
              \L F \leq 0, ~~~ \text{mod} ~~\n F,
          \end{eqnarray*}and $\n_\mu F=0$ on $\p \S_t$.
      \end{proof}

To derive the lower bound of $F$, we adopt the test function $P=\wt v F$, which is also motivated by the idea used in \cite[Propositin 3.10]{MWW} and \cite[Proposition 2.6]{WWX2}.
      \begin{prop}\label{F lower}
      Along the flow \eqref{flow with capillary}, there holds
      \begin{eqnarray*}
          F(p, t)\geq C, \quad \forall(p, t)\in M\times[0, T^{\ast}), 
      \end{eqnarray*}
      where the positive constant $C$ depends only on the initial datum. % $\Sigma_{0}$ and other known data.
      \end{prop}
      \begin{proof}
From  Proposition \ref{wvF}, the Hopf boundary lemma implies that \begin{eqnarray}\label{P}
        P:=\widetilde{v}F,
    \end{eqnarray} attains its minimum value either at $t=0$ or at some interior point of $M$. If $P$ attains its minimum value at $t=0$,  the conclusion follows directly from  \eqref{c0 est}. Assume now that $P$ attains its minimum value at some interior point, say $p_{0}\in{\rm{int}}(M)$. Then at $p_{0}$, % we have
      \begin{eqnarray*}
          \n P=0,\quad {\rm{and}}\quad \mathcal{L}P\leq 0,
      \end{eqnarray*}
   together with the expression of $\L P$ in Proposition \ref{wvF}, \eqref{formula 1} and Proposition \ref{C0 estimates},   we have
      \begin{eqnarray}\label{L P ineq}
          0\geq \mathcal{L} P\geq  v \mathcal{F} \left(\frac{1}{F}-\widetilde{v} \right)-C(F+1).
      \end{eqnarray}
  If $F\geq \frac{1}{2\widetilde{v}}$ at $p_{0}$, then the assertion follows. Otherwise, assume now that $\frac{1}{2F}> \widetilde{v}$ at $p_{0}$,   taking into account \eqref{formula 1} again and combining Proposition \ref{F upper}, we can derive $F\geq C$ by using \eqref{L P ineq}. Hence we complete the proof.% the assertion follows.
      \end{proof}
   From the expression of $F$ in \eqref{F=high order}, the lower bound of $F$ directly implies the uniform lower bound of the principal curvature of $\Sigma_{t}$. In other words, the convexity is preserved along flow \eqref{flow with capillary}.
      \begin{cor}\label{cor lower of curvature}
       Let $\S_t$ be the smooth solution of flow \eqref{flow with capillary}, then there exists a positive constant $c_{0}$ that depends on the initial datum, such that the principal curvature of $\Sigma_{t}$ satisfies
         \begin{eqnarray*}
             \min\limits_{1\leq i\leq n}\kappa_{i}(p,t)\geq c_{0},
         \end{eqnarray*}
         for all $(p, t)\in M\times [0,T^{\ast})$.
      \end{cor}
   Finally, we derive an upper bound on the mean curvature $H$ of $\Sigma_{t}$.
      \begin{prop}\label{pro H estimate}
          Along the flow \eqref{flow with capillary}, there holds 
          \begin{eqnarray*}
              H(p, t)\leq C,\quad (p,t)\in M\times[0, T^{\ast}),
          \end{eqnarray*}
          where the constant $C$ depends only on the initial datum.
      \end{prop}
      \begin{proof}
          From Proposition \ref{pro H evo}, we know that $\n_{\mu}H\leq 0$ on $\partial\Sigma_{t}$.  Thus $H$ attains its maximum value at some interior point $p_{0}\in \text{int}(M)$. Below we conduct the computation at the point $p_{0}$. 
         Let $\{e_{i}\}_{i=1}^{n}$ be an orthonormal frame around $p_0$ such that $(h_{ij})$ is diagonal. % at this point. 
         From \eqref{Y one}, we have
         \begin{eqnarray}
             \n_{i}\bar{g}(Y_{n+1},\nu)=h_{ii}\bar{g}(Y_{n+1}, e_{i})+\bar{g}(x, e_{i})\bar{g}(\nu, e^{-u}E_{n+1})-\bar{g}(e_{i}, e^{-u}E_{n+1})\bar{g}(x,\nu). \label{Y1}
         \end{eqnarray}
 The concavity of $F=\frac{H_{n}}{H_{n-1}}$ (cf. Lemma \ref{pro2.3} (2)) implies
         \begin{eqnarray}\label{concave condi}
             F^{kl,pq}\n_{i}h_{kl}\n_{i}h_{pq}\leq 0.
         \end{eqnarray}
         Substituting \eqref{Y1} and \eqref{concave condi} into \eqref{H-equ}, combining with the maximal condition and \eqref{formula 1},% we obtain
        \begin{eqnarray*}
          0 & \leq &\mathcal{L}H\\
          &\leq& -\frac{2\cos\theta} {F^{2}}h_{ii}\bar{g}(Y_{n+1}, e_{i})\n_{i}F+2
            \Big(\bar{g}(x, e_{i})-\cos\theta\bar{g}(x, e_{i})\bar{g}(\nu, e^{-u}E_{n+1})\\
            &&+\cos\theta\bar{g}(e_{i}, e^{-u}E_{n+1})\bar{g}(x,\nu)\Big)\frac{\n_{i}F}{F^{2}}-2\left(V_{0}-\cos\theta \bar{g}(Y_{n+1},\nu)\right)\frac{|\n F|^{2}}{F^{3}}\\
            &&-\frac{\left(2V_{0}-\cos\theta\bar{g}(Y_{n+1},\nu)\right)}{F}|h|^{2}+H\Big(2V_{0}+\frac{v}{F}-\cos\theta\bar{g}(Y_{n+1},\nu)\\
            &&+(V_{0}-\cos\theta\bar{g}(Y_{n+1},\nu))F^{-2}\mathcal{F}\Big)
         -nv-\frac{n}{F}\left(V_{0}-\cos\theta\bar{g}(Y_{n+1},\nu)\right)\\
         &:=&I_{1}+I_{2},
        \end{eqnarray*}
        where we use $I_{1}$ to denote all the terms involving $\n_{i} F$ and  $I_{2}$ to represent the remaining terms. To proceed, for notation simplicity we further denote
     \begin{eqnarray*}
     S&=&2\left(V_{0}-\cos\theta\bar{g}(Y_{n+1},\nu)\right),\\
         B_{i}&=&2\cos\theta h_{ii}\bar{g}(Y_{n+1},e_{i}),\\
     D_{i}&=&2\Big(\bar{g}(x, e_{i})-\cos\theta\bar{g}(x, e_{i})\bar{g}(\nu, e^{-u}E_{n+1})
            +\cos\theta\bar{g}(e_{i}, e^{-u}E_{n+1})\bar{g}(x,\nu)\Big).
     \end{eqnarray*}
     From Proposition \ref{F upper} and Proposition \ref{F lower}, we have
     \begin{eqnarray*}
        F^{3}I_{1}&=&-S|\n F|^{2}+(D_{i}-B_{i})F\n_{i}F\\
        &=&
        -S \sum_{i=1}^n  \left(\n_{i}F-\frac{(D_{i}-B_{i})}{2S}F\right)^{2}+\frac{(D_{i}-B_{i})^{2}}{4S}F^{2}\\
        &\leq & \frac{\cos^{2}\theta|h_{ii}|^{2}(\bar{g}(Y_{n+1},e_{i}))^{2}F^{2}}{2\left(V_{0}-\cos\theta\bar{g}(Y_{n+1}, \nu)\right)}+C(H+1)F.
     \end{eqnarray*}
      Combining Proposition \ref{C0 estimates}, Proposition \ref{F lower} and \eqref{ellip}, we conclude 
     \begin{eqnarray*}
       && -\frac{\left(2V_{0}-\cos\theta\bar{g}(Y_{n+1},\nu)\right)}{F}|h|^{2}+\frac{\cos^{2}\theta (\bar{g}(Y_{n+1},e_{i}))^{2}F^{-1}|h_{ii}|^{2}}{2(V_{0}-\cos\theta\bar{g}(Y_{n+1},\nu))}\\
      &\leq &\frac{|h|^{2}}{2F\left(V_{0}-\cos\theta\bar{g}(Y_{n+1},\nu)\right)}\Big(-4V_{0}^{2}+6V_{0}\cos\theta \bar{g}(Y_{n+1},\nu)-3\cos^{2}\theta(\bar{g}(Y_{n+1},\nu))^{2}\\
      &&+\cos^{2}\theta(\bar{g}(Y_{n+1}, Y_{n+1}))^{2}\Big)\\
      &=&\frac{|h_{ii}|^{2}}{2F(V_{0}-\cos\theta\bar{g}(Y_{n+1},\nu))}\Big(-3(V_{0}-\cos\theta\bar{g}(Y_{n+1},\nu))^{2}-V^{2}_{0}+\cos^{2}\theta(\bar{g}(Y_{n+1}, Y_{n+1}))^{2}\Big)\\
      &\leq &-c_{1}|h|^{2},
     \end{eqnarray*}
     for some uniform positive constant $c_{1}>0$, which only depends on the initial datum. Altogether implies  
     \begin{eqnarray*}
         0 &\leq &\mathcal{L}H\leq  -c_{1}|h|^{2}+C(H+1),
     \end{eqnarray*}
    this implies an upper bound of $H$. Hence we complete the proof.
    \end{proof}
    We obtain the uniform bound for all principal curvatures as a direct consequence of Corollary \ref{cor lower of curvature} and Proposition \ref{pro H estimate}.
\begin{cor}\label{coro curvature upper}
  Let $\S_t$ be the smooth solution of flow \eqref{flow with capillary},  then there exists a positive constant $C$ depending only on the initial datum, such that 
    \begin{eqnarray*}
        \max\limits_{1\leq i\leq n}\kappa_{i}(p,t)\leq c,
    \end{eqnarray*}
    for all $(p,t)\in M\times [0,T^{\ast})$. 
\end{cor}

Now we finish the proof of Theorem \ref{flow-result}.
 
\begin{proof}[\textbf{Proof of Theorem \ref{flow-result}}]
    From Proposition \ref{C0 estimates}, Corollary \ref{cor lower of curvature} and Corollary \ref{coro curvature upper}, we derive a uniform estimate for $\varphi$  in $C^{2}\left(\bar{\SS}^{n}_{+}\times [0, T^{\ast})\right)$ and the scalar equation \eqref{scalr flow with capillary} is uniformly parabolic. Since $|\cos\theta|<1$, the boundary value condition in \eqref{scalr flow with capillary} satisfies the uniformly oblique property. From the standard theory for the parabolic equation with oblique derivative boundary value condition (see e.g. \cite[Theorem~6.1, Theorem~ 6.4 and Theorem~6.5]{Dong}, also \cite[Theorem~5]{Ural} and \cite[Theorem~14.23]{Lie}), we obtain the uniform $C^{\infty}$-estimates and the long-time existence of solution to flow \eqref{flow with capillary}. The convergence can be shown similarly by using the argument as in  \cite[Section 3, Proposition~3.8]{SWX} or \cite[Section 3.4]{WX2022}. Hence we complete the proof.
\end{proof}

\section{The Alexandrov-Fenchel inequalities in $\HH^{n+1}$}
In this section, we obtain the Alexandrov-Fenchel inequalities for capillary hypersurface in $\HH^{n+1}$. In other words, we complete the proof of Theorem \ref{thm AF}, by applying the convergence result of flow \eqref{flow with capillary}, i.e., Theorem \ref{flow-result}.
\begin{proof}[\textbf{Proof of Theorem \ref{thm AF}}]
Assume that $\Sigma_{0}$ is strictly convex, combining with Theorem \ref{flow-result} and Proposition \ref{monotone along flow}, we prove the Theorem \ref{thm AF} for strictly convex capillary hypersurfaces in ${\H^+}$. For convex but not strictly convex capillary hypersurfaces, the inequalities hold by approximation. The equality characterization can be proved by adapting a similar argument in \cite{SWX, WX2022}. Hence we complete the proof.  
\end{proof}

\begin{proof}[\textbf{Proof of Corollary \ref{cor-Minkowski ineq}}]
The assertion follows directly from Theorem \ref{thm AF} and the expression of $\A_{2,\theta}(\wh\S)$ as
\begin{eqnarray*}
    \A_{2,\theta}(\wh\S)= \frac 1 {n(n+1)} \left(\int_\S HdA-n|\wh\S|-\sin\theta \cos\theta |\p \S|\right).
\end{eqnarray*}
\end{proof}

 We conclude this paper with a remark on the Alexandrov-Fenchel inequalities for capillary hypersurfaces in $\HH^{n+1}$. 
\begin{rem}\label{remark on general AF}
To achieve the Alexandrov-Fenchel inequalities of quermassintegrals between $\A_{k,\theta}(\wh{\Sigma})$ and $\A_{k+1,\theta}(\wh{\Sigma})$ for capillary hypersurfaces $\S$ in $\HH^{n+1}$. It is natural to design an inverse curvature flow as in \eqref{flow with cap-normal} or \eqref{flow with capillary} with the curvature function $F= \frac{H_k}{H_{k-1}}$ for $1\leq k\leq n$, instead of \eqref{F=high order}. Along such a flow, $\A_{k,\theta}(\wh{\Sigma}_t)$ is preserved, while $\A_{k+1,\theta}(\wh{\Sigma}_t)$ is monotone non-increasing with respect to time $t \geq 0$. Using the way of the maximum principle similar to Propositions \ref{F upper} and \ref{F lower}, we can establish the two-sided uniform positive bounds on the curvature function $F = \frac{H_k}{H_{k-1}}$. Additionally, we can obtain a uniform upper bound for the mean curvature $H$ as in Proposition \ref{pro H estimate}. However, the $h$-convexity preserving is not yet available for us. Nevertheless, we expect that such flow will still smoothly converge to a geodesic spherical cap, assuming that the initial capillary hypersurface is $h$-convexity (or just convexity).
\end{rem}

\noindent\textbf{Acknowledgment:} The authors would like to express their sincere gratitude to Professor Guofang Wang for his constant encouragement and many insightful discussions on this subject.

%\noindent{\textbf{Data Availability:}} Data sharing not applicable to this article as no datasets were generated or analyzed during the current study.

%\noindent\textbf{Conflict of interest:} On behalf of all authors, the corresponding author states that there is no conflict of interest.

\end{document}